\documentclass[cmp,final,envcountsect,envcountsame]{svjour}
\usepackage{amsmath,amssymb,amsfonts,tikz}
\usepackage[numbers,comma,square,sort&compress]{natbib}
\RequirePackage{color}
\RequirePackage[colorlinks,urlcolor=my-blue,linkcolor=my-blue,citecolor=my-blue]{hyperref}
\definecolor{my-blue}{rgb}{0.0,0.0,0.85}
\definecolor{my-red}{rgb}{0.5,0.0,0.0}
\definecolor{my-green}{rgb}{0.0,0.5,0.0}
\definecolor{nicos-red}{rgb}{0.75,0.0,0.0}

\newenvironment{proofof}[2]{\removelastskip\vspace{6pt}\noindent
 {\it Proof  #1.}~\rm#2}{\par\vspace{6pt}}

\numberwithin{equation}{section}

\newcommand{\be}{\begin{equation}}
\newcommand{\ee}{\end{equation}}
\newcommand{\beq}{\begin{equation}}
\newcommand{\eeq}{\end{equation}}

\newcommand{\nn}{\nonumber}
\providecommand{\abs}[1]{\vert#1\vert}

\newcommand{\fl}[1]{\lfloor{#1}\rfloor}

\def\cC{\mathcal{C}}  

\def\cG{\mathcal{G}}

\def\cM{\mathcal{M}}

\def\cS{\mathcal{S}}
\def\cI{\mathcal{I}}
\def\cK{\mathcal{K}}

\def\cL{\mathcal{L}}
 \def\cV{\mathcal{V}}
\def\xhat{\hat x}
\def\wz{\eta}
\def\esssup{\mathop{\mathrm{ess\,sup}}}
\def\gr{\mathcal G}
\font \mymathbb = bbold10 at 11pt
\def\Hbar{\bar H}

\newcommand{\one}{\mbox{\mymathbb{1}}}

\def\kS{\mathfrak{S}}

\def\bE{\mathbb{E}}
\def\bN{\mathbb{N}}
\def\bP{\mathbb{P}}
\def\bQ{\mathbb{Q}}
\def\bR{\mathbb{R}}
\def\bZ{\mathbb{Z}}

 \def\Z{\bZ} \def\Q{\bQ} \def\R{\bR}\def\N{\bN}

\def\w{\omega}
\def\om{\omega}

\def\e{\varepsilon}

\def\ddd{\displaystyle}

\def\m1{\mathbf{1}}

 \def\wt{\widetilde}   
 
\def\E{\bE}
\def\P{\bP} 

\def\funct lp{L} 
\def\funct lpbar{\bar L} 

\def\range{\mathcal R}
\def\Uset{\mathcal U}

\def\Vw{V_0}   

\def\Zpl{Z}
\def\Zpp{Z}
\def\Gpl{G}
\def\Gpp{G}
\def\gpp{g_{\text{\rm pp}}} \def\bargpp{\bar g_{\text{\rm pp}}}  \def\Lapp{g_{\text{\rm pp}}}
\def\gpl{g_{\text{\rm pl}}}  \def\Lapl{g_{\text{\rm pl}}}
\def\Bpp{B_{\text{\rm pp}}}  
\def\Bpl{B_{\text{\rm pl}}} 

\def\h{{h}}

\def\B{{B}}

\def\cE{{\mathcal E}}

\def\M{{\mathcal M}}

\DeclareMathOperator{\ri}{ri}    
\DeclareMathOperator{\aff}{aff}   

\def\lev{\sigma}  \def\rev{\tau}   
\def\rectan{\Lambda}  
\def\amatri{A}   
\def\bmatri{A} 

\def\ellc{\delta}  

\definecolor{darkgreen}{rgb}{0.0,0.5,0.0}
\definecolor{darkblue}{rgb}{0.0,0.0,0.3}
\definecolor{nicosred}{rgb}{0.65,0.1,0.1}
\definecolor{light-gray}{gray}{0.7}
\allowdisplaybreaks[1]

\begin{document}
\renewcommand\footnotemark{}

\title{Variational formulas and cocycle solutions\\ for  directed polymer and percolation models} 
\titlerunning{Variational formulas for polymers and percolation}
\author{Nicos Georgiou\inst{1} 
\and Firas Rassoul-Agha\inst{2} \and Timo Sepp\"al\"ainen\inst{3}
\thanks{F.\ Rassoul-Agha and N.\ Georgiou were partially supported by National Science Foundation grant DMS-0747758.}
\thanks{F.\ Rassoul-Agha was partially supported by National Science Foundation grant DMS-1407574 and by Simons Foundation grant 306576.}
\thanks{T.\ Sepp\"al\"ainen was partially supported by  National Science Foundation grant  DMS-1306777, by Simons Foundation grant 338287,  and by the Wisconsin Alumni Research Foundation.} 
}                     
\authorrunning{N.~Georgiou \and F.~Rassoul-Agha \and T.~Sepp\"al\"ainen}
\institute{Mathematics, University of Sussex, Falmer Campus, Brighton BN1 9QH, UK.\\ \email{N.Georgiou@sussex.ac.uk}
\and Mathematics, University of Utah,  155S 1400E,   Salt Lake City, UT 84112, USA.\\ \email{firas@math.utah.edu}
\and Mathematics, University of Wisconsin-Madison,  Van Vleck Hall, 480 Lincoln Dr.,\\  Madison WI 53706-1388, USA.\\ \email{seppalai@math.wisc.edu}
}
%
\date{Received: June 20, 2015 / Accepted: January 19, 2016}
%
%
\maketitle
\begin{abstract}
We discuss variational formulas for  the law of large numbers limits of certain models of motion in a random medium: namely, the limiting time constant for last-passage percolation  and  the limiting free energy  for directed polymers.  The results are valid for models in arbitrary dimension,    steps of the admissible paths can be general, the environment process is ergodic under spatial translations, and the potential  accumulated  along a path can depend on the environment and the next step of the path.   The variational formulas come in two types: one minimizes over gradient-like cocycles, and another one maximizes over invariant measures on the space of environments and paths.  Minimizing  cocycles can be obtained from Busemann functions when these can be proved to exist.   The results are illustrated through 1+1 dimensional  exactly solvable examples, periodic examples,  and polymers in weak disorder. 
\end{abstract}


\setcounter{tocdepth}{1}
\tableofcontents

\section{Introduction}

Existence of limit shapes has been foundational for  the study of
growth models and percolation type processes.   These limits are complicated, often coming from   subadditive sequences.
Beyond a handful of exactly solvable models, 
 very little information is available about the limit shapes.  
This article develops and studies  variational formulas for the limiting
free energies of directed random paths in a random medium,  both for positive temperature directed polymer models and for  zero-temperature
 last-passage percolation models.     
Earlier papers \cite{Ras-Sep-14}  and \cite{Ras-Sep-Yil-13} proved  variational formulas for positive temperature directed
polymers,  without addressing solutions of   these formulas.  
Article  \cite{Ras-Sep-Yil-15-} gives simpler proofs of some of the results of  \cite{Ras-Sep-Yil-13}.  

 The present paper continues the project in two directions:  

 (i) We extend the variational  formulas from positive to zero temperature, that is, we derive    
variational formulas for the limiting time constants of directed last-passage percolation models. 

(ii)  We develop an approach  for finding minimizers  for  one type of  variational formula in terms
of cocycles,  for both positive temperature  and zero temperature models.

Our paper, and the concurrent and  independent work of Krishnan  \cite{Kri-14,Kri-15-},  are the first to provide general formulas for the limits of   first- and last-passage percolation models. 

\smallskip 

The variational formulas we present  come in two types.  

\smallskip

 (a) One formula  minimizes over gradient-like cocycle  functions.  In the positive temperature case this formula mimics the commonly known  min-max   formula of  the Perron-Frobenius eigenvalue of a nonnegative matrix.  In the case of a periodic environment this cocycle variational  formula  reduces to the min-max formula from linear algebra.   The origins  of this formula go back to the  PhD thesis of Rosenbluth \cite{Ros-06}.  He adapted homogenization work  \cite{Kos-Rez-Var-06}  to deduce  a 
 formula of this type  for the quenched large deviation rate function for  random walk in
 random environment.  
 
 (b)  The second formula maximizes over invariant measures on the space of 
environments and paths.   The positive temperature version of this formula 
is of the familiar type that gives the dual of entropy as a function of the potential.  
In zero temperature the entropy disappears and only the expected potential
is left, maximized over invariant measures that are absolutely continuous 
with respect to the background measure.     In a periodic environment this zero-temperature formula reduces to the maximal average circuit weight  formula of a max-plus eigenvalue.  
 
 \smallskip 
 
The next example  illustrates the two types of variational  formulas for the   two-dimensional  corner growth model.    The notation and the details are made precise in the sequel.  

\begin{example}\label{ex:corner}
Let $\Omega=\R^{\Z^2}$ be the space of weight configurations $\w=(\w_x)_{x\in\Z^2}$ on the planar integer lattice $\Z^2$, and let $\P$ be an i.i.d.\ product probability measure on $\Omega$. Assume  $\E(\abs{\w_x}^{p})<\infty$ for some $p>2$.  Let $h\in\R^2$ be an external field parameter.   The point-to-line last-passage time is defined by 
\be \Gpl^\infty_{0,(n)}(h) =\max_{x_{0,n}: \,x_0=0}\Bigl\{ \,\sum_{k=0}^{n-1}\w_{x_k} + h\cdot x_n\Bigr\} 
	\label{gpl1.0}\ee
where the maximum is over paths  $x_{0,n}=(x_0,\dotsc, x_n)$ that begin at the origin $x_0=0$ and take   directed 
nearest-neighbor 
steps $x_k-x_{k-1}\in\{e_1, e_2\}$.  There is a law of large numbers  
\be\label{p2lh-lim1.0} 
\gpl^\infty(h)=  \lim_{n\to\infty} n^{-1}\Gpl_{0,(n)}^\infty(h)   \quad  \text{ $\P$-almost surely, simultaneously  $\forall h\in\R^2$.  }
\ee 
This  defines a deterministic convex Lipschitz function $\gpl^\infty:\R^2\to\R$.  (The subscript pl is for point-to-line and the superscript $\infty$ is for zero temperature.)  The results to be described give the following two characterizations of the limit.  

Theorem \ref{th:K-var} gives the cocycle variational  formula 
	\begin{align}
	\gpl^\infty(h)&=\inf_{F} \,\P\text{-}\esssup_\w\;  \max_{i=1,2} \; \bigl\{\w_0+ h\cdot e_i+F(\w, 0,e_i)\bigr\}.\label{eq:g:K-var1.0}
	\end{align}
 The infimum is over centered stationary  cocycles $F$. These are   mean-zero functions $F:\Omega\times(\Z^2)^2\to\R$ that satisfy   additivity $F(\w, x,y)+F(\w, y,z)=F(\w, x,z)$  and stationarity    $F(T_z\w, x,y)=F(\w, z+x,z+y)$ (Definition \ref{def:cK}).   
 
 The second formula is over measures and comes as a special case of Theorem \ref{th:g=Hstar}: 
 	\begin{align} 
	\gpl^\infty(h)&=\sup\bigl\{E^\mu[\w_0+h\cdot z]:\mu\in\cM_s(\Omega\times\{e_1,e_2\}),\,\mu\vert_{\Omega}\ll\P,\, E^\mu[\w_0^-]<\infty\bigr\}.\label{eq:g:H-var1.0} 
	\end{align}
The supremum is over probability measures 	$\mu$ on pairs $(\w, z)\in\Omega\times\{e_1,e_2\}$  that are   invariant in a natural way (described in Proposition  \ref{lm:S}) and whose $\Omega$-marginal is absolutely continuous with respect to the environment distribution $\P$.   $E^\mu$ denotes expectation under $\mu$.

 As we will see, these formulas  are valid quite generally in all dimensions, for general walks, ergodic environments,  and more complicated potentials, provided certain moment assumptions are satisfied.   
\hfill $\triangle$  
\end{example} 
 
 
In addition to deriving the formulas,  we develop a solution approach  for  the
cocycle formula in terms of stationary cocycles suitably  adapted to the potential.   Such cocycles can 
be obtained from limits of   gradients  of free energies and  last-passage times.     These
limits are called {\it Busemann functions}.  Their existence is in general a nontrivial problem.   Along the way we show that, once Busemann functions exist as almost sure limits,  their integrability  follows from the $L^1$ shape theorem which a priori is a much cruder result.

Over the last two decades  Busemann functions  have become  an   important tool  in the study of the   geometry of percolation and invariant distributions of related particle systems.   
 Study of Busemann functions is also motivated by fluctuation questions. One approach to 
 quantifying   fluctuations of 
free energy  and the paths goes through control of  fluctuations of Busemann
functions.   In 1+1 dimension these models are expected to 
lie in the Kardar-Parisi-Zhang (KPZ) universality class and there are well-supported
conjectures for universal fluctuation exponents and limit distributions.  Some of these
conjectures have been verified for a handful of exactly solvable models. 
(See surveys \cite{corw-rev, quas-icm,spoh-12,trac-wido-02}.)    In dimensions
3+1 and higher, high temperature behavior of directed polymers has been proved to be
diffusive \cite{come-yosh-aop-06}, but otherwise conjectures beyond 1+1 dimension
are murky.   

To summarize, the purpose of   this paper is to develop the variational formulas, illustrate them with examples,  and set an  agenda  for future study with the Busemann solution.  We show how the formulas work in weak disorder, in   exactly solvable 1+1 dimensional models, and in periodic  environments.   Applications  that go beyond these  cases cannot be covered within the scope of this paper  and will follow
in future work. 


Minimizing cocycles for 
  \eqref{eq:g:K-var1.0} have been  constructed
for the two-dimensional corner growth model with general i.i.d.\ weights in   \cite{geor-rass-sepp-lppbuse}.    In the sequel  \cite{geor-rass-sepp-lppgeo}  these cocycles are used to construct geodesics and to prove existence, uniqueness and coalescence  properties of directional geodesics and to study the competition interface.    In another direction of work on these formulas, article  \cite{Ras-Sep-Yil-15-}  proves the cocycle variational  formula for the annealed free energy of a directed polymer and uses it to characterize the so-called weak disorder phase of the model.

\smallskip 

 {\bf Overview of related literature.}  
Independently of the present work and with a different methodology, Krishnan \cite{Kri-14,Kri-15-} proves a variational  formula for undirected   first passage bond percolation  with bounded ergodic weights.  Taking an optimal control approach, he embeds the lattice problem into $\R^d$ and applies the recent stochastic homogenization results of Lions and Souganidis \cite{Lio-Sou-05} to derive a variational formula.  
 The resulting formula is a first passage percolation version of our formula \eqref{eq:g:K-var}.
The  homogenization parallel of our work is  \cite{Kos-Rez-Var-06, Kos-Var-08} rather than \cite{Arm-Sou-12, Lio-Sou-05}.  The quantity homogenized corresponds in our world to the finite-volume free energy.  



We run through a selection of highlights from past study of limiting shapes and  free energies.    
 For directed polymers Vargas \cite{varg-07} proved the a.s.\ existence of the limiting free energy  under moment assumptions similar to the ones we use. Earlier proofs with stronger assumptions appeared in \cite{carm-hu-02, come-shig-yosh-03}. 
In weak disorder the limiting polymer free energy is the same as the annealed one.  In strong disorder no general formulas  appeared in the literature  before  \cite{Ras-Sep-14, Ras-Sep-Yil-13}.  Carmona and Hu  \cite{carm-hu-02} gave some bounds in the Gaussian case. Lacoin \cite{laco-10}  gave small-$\beta$ asymptotics in dimensions $d=1,2$.   The earliest explicit free energy for an exactly solvable directed polymer model is the calculation in \cite{mori-oconn-07}  for the semi-discrete polymer in a Brownian environment.  Explicit   limits for the exactly solvable log-gamma polymer appear in \cite{geor-rass-sepp-yilm-15, sepp-12-aop}.   

The study of Lyapunov exponents and large deviations  for random walks in random environments is a  related   direction of literature.   \cite{Var-03-cpam,Zer-98-aop} are   two early  papers in the multidimensional setting.   

A seminal paper in the study of directed last-passage percolation is Rost 1981 \cite{rost}.  He deduced the limit shape of the corner growth model with exponential weights in conjunction  with a hydrodynamic limit for  TASEP (totally asymmetric simple exclusion process) with the step initial condition.  However, the last passage representation of this model  was discovered only later.  The study of directed last-passage percolation bloomed in the 1990s, with the first  shape results for exactly solvable cases  in \cite{aldo-diac95, cohn-elki-prop-96, jock-prop-shor-98, sepp98mprf, sepp-96}. Early motivation for  \cite{aldo-diac95}  came from   Hammersley 1972 \cite{hamm}.   The breakthroughs of \cite{baik-deif-joha-99, joha}  transformed the study of exactly solvable last-passage models and led to the first rigorous KPZ fluctuation results.  
The only universal shape result is the asymptotic result on the boundary of  $\R_+^2$ for the corner growth model by  Martin \cite{mart-04}. 


In  undirected
  first passage percolation the fundamental shape theorem  is due to  Cox and Durrett \cite{cox-durr-81}.  
A classic in the field is the flat edge result of Durrett and Liggett \cite{durr-ligg-81}.   Marchand \cite{Mar-02}  sharpened this result and   Auffinger and Damron
\cite{auff-damr-13}  built on it 
to prove differentiability of the shape at the edge of the percolation cone.  
 
 Busemann functions came on the percolation scene in the work of Newman and coauthors \cite{How-New-01, Lic-New-96, New-95}.  
 Busemann functions were shown to exist as 
almost sure limits of passage time gradients  
as a consequence  of  uniqueness and coalescence   of infinite directional geodesics, under uniform curvature assumptions on the limit shape.  These   assumptions  were  relaxed   through  a weak convergence approach of Damron and Hanson    \cite{Dam-Han-14}.  
Busemann functions have been used to study  competition in percolation models and properties of particle systems and randomly driven equations. For  a selection of the literature, see    \cite{bakh-cato-khan-14, Cat-Pim-11, Cat-Pim-12, cato-pime-13, ferr-pime-05, ferr-mart-pime-09, hoff-05, hoff-08, pime-07}. 


\medskip 

{\bf Organization of the paper.}
Section \ref{sec:free} defines the models and states the existence theorems
for the limiting free energies whose description is the purpose of the paper.  

Section \ref{sec:corr}  derives  the cocycle variational formula for the point-to-level case and develops an approach for solving these formulas. 

Section \ref{sec:tilt}  extends this to point-to-point free energy  via a duality between tilt and velocity.

Section \ref{sec:bus} demonstrates  how minimizing cocycles  arise from Busemann functions.  

Section \ref{sec:lg+exp}  explains how the theory of the paper works in explicitly solvable 1+1 dimensional models, namely the log-gamma polymer and the corner growth model with exponential weights.  

Section \ref{sec:entr}  develops variational formulas in terms of measures.   In the positive temperature case these formulas involve relative entropy.  

Section \ref{sec:finite} illustrates the results of the paper for  periodic environments  where our variational formulas become  elements of  Perron-Frobenius theory.  
\medskip 

{\bf Notation and conventions.}  We collect here some items for later reference.  $\N=\{1,2,3,\dotsc\}$, $\Z_+=\{0,1,2,\dotsc\}$,  
$\R_+=[0,\infty)$.       $\abs{x}=(\sum_i \abs{x_i}^2)^{1/2}$ 
  denotes Euclidean norm.   The standard basis vectors of $\R^d$ are 
  $e_1=(1,0,\dotsc,0), e_2=(0,1,0,\dotsc,0), \dotsc,e_d=(0,\dotsc,0,1)$. 
 $\cM_1(\mathcal X)$ denotes the space of Borel probability measures on
  a space  $\mathcal X$ and  $b\mathcal X$   the space of bounded Borel  functions   $f:\mathcal X\to\R$.       
$\P$ is a probability measure on environments $\w$, with expectation operation $\E$.  Expectation with respect to $\w$ of a multivariate function $F(\w, x, y)$ can be expressed as $\E F(\w,x,y)=\E F(x,y)=\int F(\w,x,y)\,\P(d\w)$.   $\triangle$ marks the end of an example  and  a remark.  

\section{Free energy in positive and zero temperature}\label{sec:free}  

In this section we describe the setting and state the limit theorems for 
free energy and last-passage percolation.  The positive temperature limits 
are   quoted from past work and then extended to last-passage 
percolation via a zero-temperature limit.  

Fix the dimension $d\in\N$.  Let $p:\Z^d\to[0,1]$ be a random walk probability kernel:    $\sum_{z\in\Z^d}p(z)=1$.   Assume $p$ has   finite  support  $\range=\{z\in\Z^d: p(z)>0\}$.    $\range$ must contain at least one nonzero point, and $\range$ may contain $0$.   A path  
$x_{0,n}=(x_k)_{k=0}^n$ 
in $\Z^d$ is  {\sl admissible} if its  steps satisfy 
$z_k\equiv x_k-x_{k-1}\in\range$.   The probability of an admissible path from a fixed initial point $x_0$ is $p(x_{0,n})=p(z_{1,n})=\prod_{i=1}^n p(z_i)$.    Let $\ellc=\min_{z\in\range}p(z)>0$.  

$\range$ generates the  
additive  subgroup $\gr= \{\sum_{z\in\range}a_z z:a_z\in\Z\}$ of $\Z^d$.  $\gr$   is  isomorphic to  some $\Z^k$ (Prop.~P1 on p.~65 in \cite{spitzer}).   $\Uset$ is
 the convex hull of $\range$ in $\R^d$, and $\ri\Uset$ 
 the relative interior of $\Uset$.    The common affine hull of $\range$ and $\Uset$  is denoted by $\aff\range=\aff\Uset$.
 
An {\sl environment} $\w$ is a sample point from a Polish probability space $(\Omega, \kS, \P)$ where $\kS$ is the Borel $\sigma$-algebra of $\Omega$.
$\Omega$ comes  equipped with a group  $\{T_x:{x\in\gr}\}$   of 
measurable commuting  bijections  
 that satisfy 
$T_{x+y}=T_xT_y$ and $T_0$ is the identity.
  $\P$ is a $\{T_x\}_{x\in\gr}$-invariant probability measure on $(\Omega,\kS)$. 
  This is summarized by the statement that $(\Omega,\kS,\P,\{T_x\}_{x\in\gr})$Ê
is a measurable dynamical system.   
We assume 
  $\P$  {\sl ergodic}.   As usual this means that    $\P(A)=0$ or $1$ 
for all events $A\in\kS$ that satisfy  $T_z^{-1}A=A$ for all $z\in\range$.  
Occasionally we make stronger assumptions on $\P$. 
$\E$  denotes expectation under $\P$.

A  {\sl potential}  is  a measurable function $V:\Omega\times\range^\ell\to\R$ for some $\ell\in\Z_+$, denoted by $V(\w,z_{1,\ell})$ for an environment  $\w$ and a vector   of  admissible steps $z_{1,\ell}=(z_1,\dotsc, z_\ell)\in\range^\ell$.   The case $\ell=0$ corresponds to a potential $V:\Omega\to\R$ that is a function of $\w$ alone. 
   The variational formulas from \cite{Ras-Sep-14}  and \cite{Ras-Sep-Yil-13} that
this article relies upon were proved under the following assumption on $V$.  
\begin{definition}[Class $\cL$]\label{def:cL}
A function $V:\Omega\times\range^\ell\to\R$ is in 
class $\cL$ if  for every  $ \tilde z_{1,\ell}=(\tilde z_1, \dotsc, \tilde z_\ell)\in\range^\ell$
and for every  nonzero  $z\in\range$,   
$V(\,\cdot\,,\tilde z_{1,\ell})\in L^1(\P)$   and  
\be \label{cL-cond}
\varlimsup_{\e\searrow0}\;\varlimsup_{n\to\infty} \;\max_{x\in\gr:\abs{x}\le n}\;\frac1n \sum_{0\le k\le\e n} 
\abs{V(T_{x+kz}\w, \tilde z_{1,\ell})}=0\quad\text{for $\P$-a.e.\ $\w$.}\ee  
\end{definition}
	

Membership  $V\in\cL$ depends on a combination of   mixing   of $\P$
and moments of $V$.  See   Lemma A.4 of \cite{Ras-Sep-Yil-13} for
 a precise statement.   Boundedness of $V$ is of course sufficient. 
   
  \begin{remark}\label{rm:prod}({\it Canonical settings})   Often the natural choice for $\Omega$ is a product space    $\Omega=\cS^{\Z^d}$ with a Polish space  $\cS$,  product topology, 
and 
   Borel $\sigma$-algebra $\kS$. A generic point of $\Omega$ is then denoted by 
  $\w=(\w_x)_{x\in\Z^d}$.  The mappings are shifts $(T_x\w)_y=\w_{x+y}$. 
   For example,  random weights assigned to the vertices of $\Z^d$  would be  modeled by $\Omega=\R^{\Z^d}$ and $V(\w)=\w_0$.  In fact, it would be sufficient to take $\Omega=\R^{\gr}$ since the coordinates outside $\gr$  are not needed as  long as paths begin at points in $\gr$.  
  
To represent  directed edge weights  we can  take   $\Omega=\cS^{\gr}$ with $\cS=\R^\range$ where an element $s\in\cS$ represents the   weights of the admissible edges out of the origin:  $s=(\w_{(0,z)}: z\in\range)$.    Then $\w_x=(\w_{(x,x+z)}: z\in\range)$ is the vector of edge weights out of vertex $x$. 
 Shifts act by $(T_u\w)_{(x,y)} =\w_{(x+u,y+u)}$ for $u\in\cG$. The potential is $V(\w,z)=\w_{(0,z)}=$ the weight of the edge $(0,z)$.   

To have weights on undirected nearest-neighbor edges take  
$\Omega=\R^{\cE}$ where   $\cE=\{ \{x,y\}\subset\Z^d:  \abs{y-x}=1\}$  is the set of undirected nearest-neighbor  edges on $\Z^d$. Now   $\range=\{ \pm e_i: i=1,\dotsc,d\}$, $V(\w,z)=\w_{\{0,z\}}$  and  $(T_u\w)_{\{x,y\}} =\w_{\{x+u,y+u\}}$ for $u\in\Z^d$.

   $\P$ is an {\sl i.i.d.}\ or {\sl product measure} if 
  the coordinates $\{\w_x\}_{x\in\Z^d}$ (or $\{\w_x\}_{x\in\gr}$ or $\{\w_e\}_{e\in\cE}$)  are independent and identically distributed  (i.i.d.) random
  variables under $\P$. 
      With an i.i.d.\ $\P$ and {\sl local}  
  $V$  (that is, $V$ depends on only finitely many coordinates of $\w$),  for  $V\in\cL$   it suffices to assume $V(\cdot\,, z_{1,\ell})\in L^p(\P)$ for some $p>d$ and all $z_{1,\ell}\in\range^\ell$. \hfill $\triangle$
    \end{remark}
 
For inverse temperature parameter $0<\beta<\infty$ define the $n$-step {\sl quenched partition function}  
 	\be\label{gpl8} 
	\Zpl^\beta_{0,(n)}=\sum_{x_{0,n+\ell-1}:\,x_0=0} p(x_{0,n+\ell-1}) \,e^{\beta\sum_{k=0}^{n-1}V(T_{x_k}\w,\, z_{k+1, k+\ell})}.\ee
	The sum is over admissible $(n+\ell-1)$-step  paths $x_{0,n+\ell-1}$ that start at  $x_0=0$.
The second argument of $V$ is the  $\ell$-vector  $z_{k+1, k+\ell}=(z_{k+1}, z_{k+2},\dotsc, z_{k+\ell})$ of steps, and it is not present if $\ell=0$.  	
The corresponding free energy 	is defined by 
\be\label{gpl5}  \Gpl^\beta_{0,(n)}= \beta^{-1} \log \Zpl^\beta_{0,(n)}. \ee  
In the $\beta\to\infty$ limit this turns into  the $n$-step {\sl last-passage time}  
 	\be \Gpl^\infty_{0,(n)}=\max_{x_{0,n+\ell-1}: \,x_0=0}\sum_{k=0}^{n-1}V(T_{x_k}\w,\, z_{k+1, k+\ell}).
	\label{gpl1}\ee
	 As in the definitions  above we shall consistently 
use the subscript $(n)$ with parentheses to indicate number of steps. 

 In the most basic situation where $d=2$ and  $\range=\{e_1,e_2\}$  the quantity 
   $\Gpl^\infty_{0,(n)}$ is   a {\sl point-to-line} last-passage value because admissible paths 
 $x_{0,n}$ go from $0$ to the line
$\{(i,j): i+j=n\}$.  We shall call  the general case
\eqref{gpl5}--\eqref{gpl1} {\sl point-to-level}.

The $n$-step  {\sl quenched point-to-point partition function}  is  for $x\in\Z^d$ 
 	\be\label{zpp1} \Zpp^\beta_{0,(n),x}=\sum_{x_{0,n+\ell-1}:\,x_0=0,\,x_n=x}p(x_{0,n+\ell-1})\, e^{\beta\sum_{k=0}^{n-1}V(T_{x_k}\w,\, z_{k+1, k+\ell})}\ee 
	with free energy
\[   \Gpl^\beta_{0,(n),x}= \beta^{-1} \log \Zpp^\beta_{0,(n),x}. \]
Its zero-temperature limit is 
 the $n$-step {\sl point-to-point last-passage time} 
 	\be\Gpp^\infty_{0,(n),x}=\max_{x_{0,n+\ell-1}:\, x_0=0,\,x_n=x}\sum_{k=0}^{n-1}V(T_{x_k}\w,\, z_{k+1,k+\ell}).\label{gpp1}\ee
	
\begin{remark}\label{path-rm}	
The formulas for limits presented in this paper are for the case where the 
length of the path is restricted, as in \eqref{zpp1} and \eqref{gpp1}, so that  
only those paths  that reach $x$ from $0$ in exactly  $n$ steps are considered.  This is indicated
by  the subscript $(n)$. 	Extension  to 
 paths of unrestricted length   from $0$ to $x$  or from $0$  to a  hyperplane is left for  future work.  
   In the most-studied  directed models   this restriction can be dropped because 
each path between two given points has the same number of steps.   
 Examples where this is the case 
  are  $\range=\{e_1,\dotsc,e_d\}$  and $\range=\{(z',1): z'\in\range'\}$
 for a finite  subset $\range'\subset\Z^{d-1}$.      \hfill   $\triangle$
\end{remark}


To take limits of point-to-point quantities    we 
specify  lattice points  $\xhat_n(\xi)$ that approximate $n\xi$ for $\xi\in\Uset$. 
For each  point $\xi\in\Uset$   fix weights $\alpha_z(\xi)\in[0,1]$ such that 
$\sum_{z\in\range}\alpha_z(\xi) =1$  and 
$\xi=\sum_{z\in\range}\alpha_z(\xi) z$.  Then define a path 
\begin{align}\label{eq:def:xhat}
\xhat_n(\xi)=\sum_{z\in\range}\bigl(\lfloor n\alpha_z(\xi)\rfloor +b_z^{(n)}(\xi)\bigr) z, \quad n\in\Z_+,  
\end{align}
where  $b_z^{(n)}(\xi)\in\{0,1\}$ are  arbitrary but subject to these constraints:  
 if $\alpha_z(\xi)=0$ then  $b_z^{(n)}(\xi)=0$,  and $\sum_{z\in\range} b_z^{(n)}(\xi) = n-\sum_{z\in\range}\fl{n\alpha_z(\xi)}$.   In other words,  $\xhat_n(\xi)$ is a lattice point 
that approximates $n\xi$ to within a constant independent of $n$, can be reached in $n$ $\range$-steps   from the origin,
and uses only those steps that appear in the pre-specified   convex representation 
$\xi=\sum_z \alpha_z z$.   
When $\xi\in\Uset\cap\Q^d$ we require that  $\alpha_z(\xi)$ be rational. This is 
  possible by Lemma A.1  of \cite{Ras-Sep-Yil-13}.

The next   theorem defines the limits whose study is the purpose of the paper.  
We state it so that it covers simultaneously both the positive temperature 
($0<\beta<\infty$) and the zero-temperature case (last-passage percolation, 
or $\beta=\infty$).   The subscripts are pl for point-to-level and pp for point-to-point.

\begin{theorem}\label{th:p2p}
Let  $V\in\cL$ and assume $\P$ ergodic. 
Let $\beta\in(0,\infty]$.   

\smallskip 

{\rm (a)}    The  nonrandom  limit 
		\begin{align}\label{eq:g:p2l}
		 \gpl^\beta=\lim_{n\to\infty}n^{-1}\Gpl^\beta_{0,(n)}
		\end{align}
exists $\P$-a.s.\ in $(-\infty,\infty]$.

\smallskip

{\rm (b)}   There exists an event $\Omega_0$ with $\P(\Omega_0)=1$ such that the following holds for all $\w\in\Omega_0$.  
For all $\xi\in\Uset$ and any choices made in the definition of $\xhat_n(\xi)$ in \eqref{eq:def:xhat},   the limit 
	\begin{align}\label{eq:g:p2p}
	\gpp^\beta(\xi)=\lim_{n\to\infty}n^{-1}\Gpp^\beta_{0,(n),\xhat_n(\xi)}
	\end{align}
 exists in $(-\infty,\infty]$.  For a particular $\xi$ the  limit is  independent of the choice of  convex representation $\xi=\sum_z\alpha_z(\xi) z$ and  the numbers $b^{(n)}_z(\xi)$ that define
 $\xhat_n(\xi)$ in \eqref{eq:def:xhat}.   We have the almost sure identity 
 	\begin{align}\gpl^\beta
	=\sup_{\xi\in\Q^d\cap\Uset}\gpp^\beta(\xi)
	=\sup_{\xi\in\Uset}\gpp^\beta(\xi).\label{eq:sup p2p:lpp}
	\end{align}
 \end{theorem}



\begin{proof}   The case $0<\beta<\infty$ is covered by Theorem 2.2 of \cite{Ras-Sep-14}.  (The kernel there is the uniform one $p(z)=\abs{\range}^{-1}$ but this makes no difference to the arguments.  Alternatively, the kernel can be moved into the potential.)  

 For any $0<\beta<\infty$, 
\begin{align*}
\Gpl^\infty_{0,(n)}+\beta^{-1}(n+\ell-1)\log\ellc\;&\le\;\beta^{-1}\log\Zpl^\beta_{0,(n)}\;\le\; \Gpl^\infty_{0,(n)}\\
\text{ and }\qquad 
\Gpp^\infty_{0,(n),x}+\beta^{-1}(n+\ell-1)\log\ellc\;&\le\;\beta^{-1}\log\Zpp_{0,(n),x}^\beta
\;\le\;\Gpp^\infty_{0,(n),x}.
\end{align*}
Divide by $n$,  let first $n\to\infty$ and then $\beta\to\infty$.  This
gives the existence of the limits for the case $\beta=\infty$.  
 We also get these bounds, uniformly in $\w$ and $\xi\in\Uset$:   
\be \label{eq:g-Lambda} \begin{aligned}
\gpl^\infty+\beta^{-1}\log\ellc\;&\le\; \Lapl^\beta\;\le\; \gpl^\infty
\\
\text{ and}\qquad 
\gpp^\infty(\xi)+\beta^{-1}\log\ellc\;&\le\; \Lapp^\beta(\xi)\;\le\; \gpp^\infty(\xi).       
\end{aligned}   \ee
These bounds extend  \eqref{eq:sup p2p:lpp} from $0<\beta<\infty$ to $\beta=\infty$.   \hfill\hfill \qed  
 \end{proof}

Since our hypotheses are fairly general, we need to address 
 the randomness, finiteness, and 
regularity  of  the limits.   For $0<\beta<\infty$  the remarks below repeat 
claims proved in \cite{Ras-Sep-14}.  The properties extend to $\beta=\infty$
by way of bounds \eqref{eq:g-Lambda} as $\beta\to\infty$.  

\begin{remark}($\P$ {\it ergodic})
If we only assume $\P$ ergodic and place no further restrictions on admissible paths
then we need to begin by assuming that $\gpl^\beta\in\R$.  An obvious way to
guarantee this would be to assume that $V$ is bounded above (in addition to what
is assumed to have $V\in\cL$). 
Under the assumption $\gpl^\beta\in\R$
the point-to-point limit 
  $\gpp^\beta(\xi)$  
  is a nonrandom, real-valued, concave and  continuous function 
on the relative interior $\ri\Uset$.    Boundary values $\gpp^\beta(\xi)$ 
for $\xi\in\Uset\smallsetminus\ri\Uset$ can be
random, but 
on the whole of $\Uset$,   for  $\P$-a.e.~$\w$,   the (possibly random) function   $\xi\mapsto \gpp^\beta(\xi;\w)$ is    lower semicontinuous and bounded.   The    upper semicontinuous regularization of $\gpp^\beta$  and  its unique continuous extension from  
 $\ri\Uset$ to   $\Uset$ are equal  and nonrandom.   \hfill $\triangle$
 \end{remark}

\begin{remark}\label{rmk:dir-iid}({\it Directed i.i.d.\ $L^{d+\e}$ case})
Assume the 
canonical setting from Remark \ref{rm:prod}:  $\Omega$ is a product space,  $\P$ is i.i.d.,   $V$ is local,
 and  $\E[\abs{V(\w,z_{1,\ell})}^p]<\infty$ 
  for some $p>d$ and $\forall z_{1,\ell}\in\range^\ell$.     Assume additionally that $0\not\in\Uset$.   We call this 
   the {\sl directed i.i.d.\ $L^{d+\e}$  case}.   Then $V\in\cL$,  $\gpl^\beta\in\R$, and 
the point-to-point limit 
  $\gpp^\beta(\xi)$  
  is a nonrandom, real-valued, concave and  continuous function 
on all of  $\Uset$ (Theorem 3.2(a) of \cite{Ras-Sep-14}).     \hfill $\triangle$
\end{remark}

\section{Cocycle variational formula for the point-to-level case}
\label{sec:corr} 

 In  Sections \ref{sec:corr}--\ref{sec:bus}   we study   potentials
 of the form 
 \be\label{V0h} V(\w, z)=\Vw(\w, z)+h\cdot z, \qquad  (\w, z) \in \Omega\times\range \ee
  for a measurable function 
 $\Vw:\Omega\times\range\to\R$ and  a vector $h\in\R^d$.  We think of $V_0$ as fixed and
 $h$ as  a variable and hence amend our notation as follows.  As before the steps of admissible paths are $z_k=x_k-x_{k-1}\in\range$. 
  \be
\Gpl_{0,(n)}^\beta(h)=   \beta^{-1}\log
 \sum_{x_{0,n}: \,x_0=0} p(x_{0,n})\,e^{\beta \sum_{k=0}^{n-1}\Vw(T_{x_k}\w, \,z_{k+1}) + \beta h\cdot x_n}   
 \label{p2lh-beta}\ee
 for $0<\beta<\infty$,
 \be
\Gpl_{0,(n)}^\infty(h)=  
 \max_{x_{0,n}: \,x_0=0} \Bigl\{ \;\sum_{k=0}^{n-1}\Vw(T_{x_k}\w, z_{k+1}) + h\cdot x_n \Bigr\},  
 \label{p2lh}\ee
and 
\be\label{p2lh-lim} 
\gpl^\beta(h)=  \lim_{n\to\infty} n^{-1}\Gpl_{0,(n)}^\beta(h)   \quad  \text{a.s.  for all  $0<\beta\le \infty$}.
\ee 
Limit  \eqref{p2lh-lim}  is a  special case of   \eqref{eq:g:p2l}.    

 By \eqref{eq:g-Lambda}, if $\gpl^\beta(0)$  is finite for one $\beta\in(0,\infty]$, it is finite for all $\beta\in(0,\infty]$.    This can be guaranteed
  by assuming $\Vw$ bounded above, or by the directed 
i.i.d.\ $L^{d+\e}$ assumption of
Remark \ref{rmk:dir-iid}, or by some other case-specific assumption. 
If $\gpl^\beta(0)$  is finite, it is clear from the expressions above that  $\gpl^\beta(h)$ is 
 a real-valued convex Lipschitz function of $h\in\R^d$.  
 
We   develop   a variational formula for $\gpl^\beta(h)$ for  $\beta\in(0,\infty]$  
 in terms of gradient-like 
cocycles, and identify a condition that singles out extremal cocycles.    
For $0<\beta<\infty$ this variational formula appeared in \cite{Ras-Sep-Yil-13}
and here we extend it to $\beta=\infty$.  The solution proposal is new for 
all $\beta$.  
  
\begin{definition}[Cocycles]
\label{def:cK}
A measurable function $F:\Omega\times\gr^2\to\R$ is   a {\rm stationary  cocycle}
if it satisfies these two   conditions for $\P$-a.e.\ $\w$ and  all $x,y,z\in\gr$:
\begin{align*}
F(\w, z+x,z+y)&=F(T_z\w, x,y)  \qquad \text{{\rm(}stationarity{\rm)}}  \\
F(\w,x,y)+F(\w,y,z)&=F(\w, x,z) \qquad  \;\  \ \text{{\rm(}additivity{\rm)}.}
\end{align*}
If   $\E\abs{F(x,y)} <\infty$   $\forall x,y\in\gr$ then $F$ is an  {\rm   $L^1(\P)$ cocycle}, and if also $\E[F(x,y)]=0$  $\forall x,y\in\gr$ then $F$ is  {\rm centered}.    $\cK$ denotes the space of stationary $L^1(\P)$  cocycles, and $\cK_0$ denotes  the subspace of  centered stationary $L^1(\P)$  cocycles.  
\end{definition}



As illustrated above,   $\w$ can be   dropped from the notation
$F(\w,x,y)$.  
The term cocyle is borrowed from differential forms terminology, see e.g.\  \cite{Ken-09}.  One could also use the term {\sl conservative flow} or {\sl curl-free flow} following 
vector fields terminology. 

The space  $\cK_0$  
is the $L^1(\P)$ closure of gradients $F(\w,x,y)=\varphi(T_y\w)-\varphi(T_x\w)$ \cite[Lemma C.3]{Ras-Sep-Yil-13}.   
For $B\in\cK$    there exists a 
 vector $h(\B)\in\R^d$ such that   
 \be \E[\B(0,z)]=-h(\B)\cdot z  \qquad \text{ for all $z\in\range$. }  \label{EB}\ee
   Existence of $h(B)$ follows  because $c(x)=\E[\B(0,x)]$ is an additive function  on the 
  group $\gr\cong\Z^k$.   $h(B)$ is not unique unless $\range$ spans $\R^d$, but the inner products $h(B)\cdot x$ for $x\in\gr$ are uniquely defined.  
Then 
\be F(\w, x,y)=  h(\B)\cdot (x-y)-\B(\w, x,y), \qquad x,y\in\gr \label{FF}\ee
 is a centered  stationary $L^1(\P)$  cocycle.



\begin{theorem}\label{th:K-var}
Let  $\Vw\in\cL$ and assume $\P$ ergodic.
Then the limits in \eqref{p2lh-lim}   have these variational representations:   for $0<\beta<\infty$ 
\be	\gpl^\beta(h)=\inf_{F\in\cK_0} \,\P\text{-}\esssup_\w\;  \beta^{-1}\log \sum_{z\in\range} p(z)e^{\beta \Vw(\w,z)+\beta h\cdot z+\beta F(\w, 0,z)} \label{eq:Lambda:K-var}  \ee
 and 
	\begin{align}
	\gpl^\infty(h)&=\inf_{F\in\cK_0} \,\P\text{-}\esssup_\w\;  \max_{z\in\range} \{\Vw(\w,z)+ h\cdot z+F(\w, 0,z)\}.\label{eq:g:K-var}
	\end{align}
A minimizing 	$F\in\cK_0$ exists for each $0<\beta\le\infty$ and $h\in\R^2$. 
 \end{theorem}

\begin{proof}   
Theorem 2.1 of  \cite{Ras-Sep-Yil-15-}
gives formula \eqref{eq:Lambda:K-var}   for $0<\beta<\infty$.  The kernel in that reference is the uniform one $p(z)=\abs\range^{-1}$ but changing the kernel makes no difference to the proof.  
To get the formula for $\beta=\infty$,  note that for $\beta>0$ and  $F\in\cK_0$, 
 \begin{align*} 
 &\beta^{-1}\log\sum_z p(z)e^{\beta V(\w,z)+\beta F(\w,0,z)}
 \le\max_z \{ V(\w,z)+ F(\w,0,z)\} \\
 &\qquad\qquad\qquad\qquad\qquad\qquad\le
 \beta^{-1}\log\sum_z p(z)e^{\beta V(\w,z)+\beta F(\w,0,z)}+\beta^{-1}\log\ellc^{-1}.\end{align*}
Thus  
 \begin{align*}
 \Lapl^\beta \;\le\; \inf_{F\in\cK_0} \P\text{-}\esssup_\w\;    \max_z \{V(\w,z)+F(\w,0,z)\} \le \; \Lapl^\beta+\beta^{-1}\log\ellc^{-1}.\end{align*} 
 Formula \eqref{eq:g:K-var}  follows from this and \eqref{eq:g-Lambda}, upon letting
 $\beta\to\infty$.
  Theorem 2.3 of \cite{Ras-Sep-Yil-15-}  gives the existence of a minimizer   for $0<\beta<\infty$, and the same proof works   also for  $\beta=\infty$. \hfill\hfill\qed
 \end{proof}
 
Assuming  $\gpl^\beta(0)$ finite is not necessary  for Theorem \ref{th:K-var}.  By the assumption $\Vw\in\cL$, any $F\in\cK_0$ that makes the right-hand side of \eqref{eq:g:K-var}  finite  satisfies   the ergodic theorem (Theorem \ref{th:Atilla}) in the appendix.  Then potential $V(\w,z)$ can be replaced by $V(\w,z)+F(\w,0,z)$  without altering $\gpl^\beta(h)$,  and consequently $\gpl^\beta(h)$  is finite.
 
   Formulas \eqref{eq:Lambda:K-var} and \eqref{eq:g:K-var} can be viewed as infinite-dimensional versions of the min-max variational formula for the Perron-Frobenius eigenvalue of a nonnegative matrix.   This connection is discussed 
in Section \ref{sec:finite}.

 The next definition and  theorem offer a way  
to identify  a minimizing $F$ 
for   \eqref{eq:Lambda:K-var} and \eqref{eq:g:K-var}.   Later we explain how Busemann functions provide minimizers that match this recipe.  
    That this approach is feasible will be demonstrated by examples:   weak disorder (Example \ref{ex:weak1}),  the exactly solvable log-gamma polymer 
   (Section \ref{sec:lg} below) and the corner growth model with exponential weights
(Section \ref{sec:exp}).  This strategy is  carried out for the two-dimensional
  corner growth model with general weights (a non-solvable case)  in article  \cite{geor-rass-sepp-lppbuse}.  

\begin{definition}\label{def:bdry-model}
Fix $\beta\in(0,\infty]$.  
A stationary $L^1$ cocycle $B$ 
{\rm is adapted  to}  potential $\Vw$ if 
the following condition  holds.   If $0<\beta<\infty$ the requirement is 
\be 
\sum_{z\in\range}p(z)\,e^{\beta \Vw(\w,z)-\beta B(\w, 0,z)} =1 
 \quad\text{ for }\P\text{-a.e.\ }\w,  
\label{VBbeta}\ee
while if $\beta=\infty$ then the condition is  
\be 
\max_{z\in\range} \{ \Vw(\w,z) -B(\w, 0,z)\}=0 
\quad\text{ for }\P\text{-a.e.\ }\w.
\label{VB}\ee
 \end{definition}
 

\begin{theorem}\label{thm:minimizer}   Fix $\beta\in(0,\infty]$, 
assume $\P$   ergodic and  $\Vw\in\cL$.  
Suppose we have  a stationary $L^1$ cocycle   $B$  that is adapted to 
  $\Vw$ in the sense of Definition \ref{def:bdry-model}.
Define $h(B)$ and $F$ as in   \eqref{EB}--\eqref{FF}. Then we have  conclusions {\rm (i)--(ii)} below.

\smallskip

{\rm (i)}  $\gpl^\beta(h(B))=0$.    $\gpl^\beta(h)$ is finite for all $h\in\R^d$. 

\smallskip

{\rm (ii)}  $F$ solves the variational formula.  Precisely, assume $h\in\R^d$    satisfies 
\be  (h-h(B))\cdot(z-z')=0 \quad\text{ for all $z,z'\in\range$. } \label{h(B)-1}\ee
Under this assumption  we have the two cases below. 

\smallskip

{\rm (ii-a)}   Case $0<\beta<\infty$.    $ F$ is a minimizer in \eqref{eq:Lambda:K-var} 
  for potential  $V(\w,z)=\Vw(\w,z)+h\cdot z$.  The essential supremum in  \eqref{eq:Lambda:K-var} 
   disappears and we have, 
 for $\P$-a.e.\ $\w$ and any $z'\in\range$, 
\be \label{eq:Kvar:minbeta}	\begin{aligned}
	\Lapl^\beta(h)&\;=\;\beta^{-1}\log\sum_{z\in\range} p(z)\,e^{\beta\Vw(\w,z)+\beta h\cdot z+\beta F(\w, 0,z)}
	\;=\;(h-h(B))\cdot z'.
	\end{aligned}\ee
	
\smallskip

{\rm (ii-b)}  Case $\beta=\infty$.  Then     $F$ is a minimizer in  
 \eqref{eq:g:K-var} for potential  $V(\w,z)=\Vw(\w,z)+h\cdot z$.  The essential supremum in  \eqref{eq:g:K-var} 
   disappears   and we have,  
 for $\P$-a.e.\ $\w$ and any $z'\in\range$, 
\be \label{eq:Kvar:min}	\begin{aligned}
	\gpl^\infty(h)&\;=\;\max_{z\in\range} \;  \{\Vw(\w,z)+h\cdot z+F(\w, 0,z)\}
	\;=\;(h-h(B))\cdot z'.
	\end{aligned}\ee
\end{theorem}

Condition \eqref{h(B)-1}  says that $h-h(B)$ is orthogonal to the affine hull of
$\range$ in $\R^d$.  If $0\in\Uset$   this affine hull is  the linear  span of $\range$ in  $\R^d$.  

\begin{remark}({\it Correctors})
A mean-zero cocycle that minimizes in  \eqref{eq:Lambda:K-var}  or  \eqref{eq:g:K-var}  without the essential supremum (that is, satisfies the first equality of \eqref{eq:Kvar:minbeta}  or \eqref{eq:Kvar:min}) could be  called a {\it corrector} by analogy with the homogenization literature (see  for example Section 7 in \cite{Kos-07} and top of  page 468 in  \cite{Arm-Sou-12}).    These correctors have been useful in the study of infinite geodesics in the corner growth model   \cite{geor-rass-sepp-lppgeo} 
and infinite directed  polymers  \cite{geor-rass-sepp-yilm-15}.        \hfill $\triangle$
\end{remark}

\begin{proofof}{of Theorem \ref{thm:minimizer}}
{{\sl Case $0<\beta<\infty$.}}   From assumption \eqref{VBbeta} and definition \eqref{FF} of $F$
\begin{align}
\log\sum_{z\in\range} p(z)\,e^{\beta \Vw(\w,\,z)+\beta h(B)\cdot z+ \beta F(\w, 0,z)}=0\quad \text{ for $\P$-a.e.\ $\w$}.\label{eq:auxbeta}
\end{align}
Iterating this 
gives (with $x_k=z_1+\dotsm+z_k$) 
\begin{align}
\log\sum_{z_{1,n}\in\range^n}p(x_{0,n})\,e^{\beta \sum_{k=0}^{n-1} \Vw(T_{x_k}\w, \,z_{k+1})+\beta h(B)\cdot x_n+\beta   F(\w, 0, x_{n})}=0.\label{eq:stat-tiltbeta}
\end{align}

Assumption \eqref{VBbeta}  gives the bound 
$  F(\w, 0,z) \le  \Vw^*(\w)+C  $  for $z\in\range$, 
 with 
 \be\label{V*} \Vw^*(\w)=\max_{z\in\range} \abs{\Vw(\w,z)} \ee
 that satisfies  $\Vw^*\in\cL$  and   a constant $C$.  
By Theorem \ref{th:Atilla} in the appendix,   $F(\w, 0, x_{n})=o(n)$ uniformly in $z_{1,n}$, $\P$-almost surely. 
 It follows from \eqref{eq:stat-tiltbeta} that $\Lapl^\beta(h(B))=0$.   Since the steps of the walks are bounded,  finiteness of $\gpl^\beta(h)$  for all $h$ follows from the definition \eqref{p2lh-beta}.  

Assume   \eqref{h(B)-1}  for $h$.  Then   \eqref{eq:auxbeta} gives
\[ \beta^{-1} \log\sum_{z\in\range}p(z)\, e^{\beta \Vw(\w,\,z)+\beta h\cdot z+ \beta F(\w, 0,z)}=  (h-h(B))\cdot z' \]
while from   \eqref{p2lh-beta} and \eqref{p2lh-lim} 
\[ \Lapl^\beta(h)=\Lapl^\beta(h(B)) + (h-h(B))\cdot z' =  (h-h(B))\cdot z'. \]
We have verified  \eqref{eq:Kvar:minbeta}.

\smallskip

{\sl Case $\beta=\infty$.}
From assumption \eqref{VB} and definition \eqref{FF} of $F$
\begin{align}
\max_{z\in\range} \{\Vw(\w,z)+h(B)\cdot z+ F(\w, 0,z)\}=0\quad \text{ for $\P$-a.e.\ $\w$}.\label{eq:aux}
\end{align}
Iterating this gives (with $x_k=z_1+\dotsm+z_k$) 
\begin{align}
\max_{z_{1,n}\in\range^n}\Bigl\{\;\sum_{k=0}^{n-1} \Vw(T_{x_k}\w, z_{k+1})+h(B)\cdot x_n+F(\w, 0, x_{n}) \Bigr\}=0.\label{eq:stationary-tilt}
\end{align}
By Theorem \ref{th:Atilla},  $F(\w, 0, x_{n})=o(n)$ uniformly in $z_{1,n}$ $\P$-a.s. 
It follows that $\gpl(h(B))=0$. 

Assume   \eqref{h(B)-1}  for $h$.  Then   \eqref{eq:aux} gives
\[  \max_{z\in\range} \{\Vw(\w,z)+h\cdot z+ F(\w, 0,z)\}=  (h-h(B))\cdot z' \]
while from   \eqref{p2lh}--\eqref{p2lh-lim} 
\[ \gpl^\infty(h)=\gpl^\infty(h(B)) + (h-h(B))\cdot z' =  (h-h(B))\cdot z'.\tag*{\qed}\] 
\end{proofof}



\smallskip

\begin{remark}  The  results of this section   extend
 to the   more general potentials   $V(T_{x_k}\w,z_{k+1,k+\ell})$ discussed
 in Section \ref{sec:free}.   For the  definition of the cocycle  see 
  Definition 2.2 of \cite{Ras-Sep-Yil-13}.  
We do not pursue these generalizations  to avoid becoming overly technical and 
because presently we do not have an interesting
example of  this more general potential.       \hfill $\triangle$
\end{remark}

The remainder of this section discusses an  example that illustrates  Theorem \ref{thm:minimizer}.  

\begin{example}\label{ex:weak1}({\it Directed polymer in weak disorder})   
We consider the standard  $k+1$ dimensional  directed polymer in an i.i.d.\ random environment, or ``bulk disorder''. (For references see \cite{come-shig-yosh-04, come-yosh-aop-06, denholl-polymer}.) We show that the condition of weak disorder itself gives the 
corrector that solves the variational formula for the point-to-level free energy.  
The background walk is a simple random walk in $\Z^k$, and we use an additional 
$(k+1)$st coordinate to represent time.  
So $d=k+1$,  $\Omega=\R^{\Z^d}$,  $\P$ is  i.i.d.,   $\range=\{(\pm e_i,1): 1\le i\le k\}$ and $p(z)=\abs{\range}^{-1}$ for $z\in\range$. The potential is 
simply the environment at the site:  $V_0(\w)=\w_0$.  
  
 Define the   logarithmic moment generating functions 
\beq \lambda(\beta)=\log\E(e^{\beta\w_0}) \quad\text{for $\beta\in\R$}  \label{lambdabeta}\eeq
and  
\beq \kappa(h)=\log \sum_{z\in\range} p(z)\, e^{h\cdot z} 
\quad\text{for $h\in\R^{d}$.}  
\label{kappa1}\eeq
Consider only $\beta$-values such  that $\lambda(\beta)<\infty$.  
 The normalized partition function 
\[  
W_n = e^{-n(\lambda(\beta)+\kappa(\beta h))} \sum_{x_{0,n}} p(x_{0,n})\,e^{\beta \sum_{k=0}^{n-1}\w_{x_k} + \beta h\cdot x_n}
\]
is a positive   mean 1 martingale.   The weak disorder assumption is this:
\be\label{weak:ui}
\text{the martingale  $W_n$ is uniformly integrable. } 
\ee
Given $h\in\R^d$, this can be guaranteed by taking  $k\ge 3$ and  small enough $\beta>0$ 
(see Lemma 5.3 in \cite{Ras-Sep-12-arxiv}).  
Then $W_n\to W_\infty$ a.s.\ and in $L^1(\P)$,  $W_\infty\ge 0$ and   $\E W_\infty=1$.  
The event $\{W_\infty>0\}$ is a tail event in the product space of environments, 
and hence by Kolmogorov's 0-1 law
we must have $\P(W_\infty>0)=1$. 
This gives us the limiting point-to-level  free energy:  
\be  \label{weak:g} \begin{aligned}
\gpl^\beta(h)  &=  \lim_{n\to\infty} n^{-1} \beta^{-1}\log
 \sum_{x_{0,n}} p(x_{0,n})\,e^{\beta \sum_{k=0}^{n-1} \w_{x_k} + \beta h\cdot x_n}\\
 &=   \lim_{n\to\infty} n^{-1} \beta^{-1}\log  W_n 
+   \beta^{-1}(\lambda(\beta)+\kappa(\beta h) )   \\
&= \beta^{-1}(\lambda(\beta)+\kappa(\beta h) )  . 
\end{aligned} \ee
Decomposition according to the first step (Markov property) gives 
\[  W_n(\w) = e^{-\lambda(\beta)-\kappa(\beta h)} \sum_{z\in\range}  p(z) 
e^{\beta \w_{0} + \beta h\cdot z} W_{n-1}(T_z\w) 
\]
and a passage to the limit  
\be\label{weak:W} 
 W_\infty(\w) = e^{-\lambda(\beta)-\kappa(\beta h)} \sum_{z\in\range}  p(z) 
e^{\beta \w_{0} + \beta h\cdot z} W_{\infty}(T_z\w) \quad \text{$\P$-a.s.}   
\ee  
Combining \eqref{weak:g} and \eqref{weak:W} gives 
\be\label{weak:g3} 
\gpl^\beta(h) =  \beta^{-1}\log \sum_{z\in\range} p(z)e^{\beta \w_0+\beta h\cdot z+\beta F(\w, 0,z)}  \quad \text{$\P$-a.s.} 
\ee
with the gradient 
\be\label{weak:F} 
F(\w, x,y)=\beta^{-1} \log W_{\infty}(T_y\w)  -  \beta^{-1} \log W_{\infty}(T_x\w).  
\ee
In order to check that  $F$ is a centered cocycle it remains to verify that $F(\w, 0, z)$ is integrable
and mean-zero. 
Equation \eqref{weak:g3} gives an upper bound that shows  $\E[F(\w,0,z)^+]<\infty$. 
We argue indirectly that also $\E[F(\w,0,z)^-]<\infty$.  The first   limit
in probability below comes from stationarity. 
\begin{align*}
0\; &\overset{\text{prob}}=\lim_{n\to\infty} \;\bigl[ n^{-1}\beta^{-1}\log W_{\infty}(T_{nz}\w)  -  n^{-1} \beta^{-1}\log W_{\infty}(\w)\bigr]\\
&=  \lim_{n\to\infty} \, n^{-1}\sum_{k=0}^{n-1}  F(T_{kz}\w, 0, z)  . 
\end{align*}  
Since  $\E[F(\w,0,z)^+]<\infty$,  the assumption  $\E[F(\w,0,z)^-]=\infty$  and the ergodic theorem
would force the limit above to $-\infty$.  Hence it must be that $F(\w, 0, z)\in L^1(\P)$.
The limit above then gives $\E[F(\w,0,z)]=0$.   

To summarize, \eqref{weak:g3} shows that the centered cocycle $F$ satisfies  \eqref{eq:Kvar:minbeta} for $V(\w,z)=\w_0+h\cdot z$ for  this  particular  value $(\beta,h)$.      $F$ is the corrector given in  Theorem \ref{thm:minimizer}, from  the cocycle $B$  that is adapted to 
$V_0$ given by  
\[  B(\w, x,y)=  \gpl^\beta(h) e_d\cdot (y-x)- h\cdot (y-x) -F(\w, x,y) \]
with 
$ h(B)= h-   \gpl^\beta(h) e_d. $   A vector  $\tilde h$  satisfies \eqref{h(B)-1}  if and only if 
$\tilde h=h+\alpha e_d$ for some $\alpha\in\R$. The conclusion of the theorem,  that $F$ is a corrector for potential   $ \Vw(\w)= \w_0$ and all such tilts   $\tilde h$, is  obvious because $e_d\cdot x_n=n$ for admissible paths. 
\hfill $\triangle$ 

 \end{example}

\section{Tilt-velocity duality} \label{sec:tilt} 
Section \ref{sec:corr} gave a variational description of  the point-to-level limit 
 in terms of stationary cocycles.    Theorem \ref{th:K-var-pp} below   extends this description to point-to-point
 limits  via tilt-velocity duality.    Tilt-velocity duality is the familiar idea   from large 
 deviation theory that   pinning the path is dual to 
  tilting the energy by an   external field.    In the positive temperature setting this is exactly the 
 convex duality of the quenched large deviation principle for the endpoint of the 
 path (see Remark 4.2  in \cite{Ras-Sep-14}).  


We continue to consider potentials of the form 
$V(\w,z)=\Vw(\w,z)+h\cdot z$   in general dimension $d\in\N$, 
with   $\beta\in(0,\infty]$ and  $\P$  ergodic.   
As above, the point-to-level  limits $\gpl^\beta(h)$ are defined by  \eqref{p2lh-lim}. 
For the point-to-point limits $\gpp^\beta(\xi)$ we use only the $\Vw$-part
of the potential.  So  for $\xi\in\Uset$   define 
  \be
\gpp^\beta(\xi)=  \lim_{n\to\infty} n^{-1} \beta^{-1}\log
 \sum_{\substack{x_{0,n}:\, x_0=0,\\ \quad x_n=\xhat_n(\xi)}} p(x_{0,n})\,e^{\beta \sum_{k=0}^{n-1}\Vw(T_{x_k}\w, \,z_{k+1})  }  
 \label{p2pV_0}\ee
 for $0<\beta<\infty$ and 
\be
\gpp^\infty(\xi)=  \lim_{n\to\infty}  
 \max_{\substack{x_{0,n}:\, x_0=0,\\ \quad x_n=\xhat_n(\xi)}} 
\;  n^{-1}\sum_{k=0}^{n-1}\Vw(T_{x_k}\w, z_{k+1})  . 
 \label{p2pV_0LPP}\ee

  In this context we call the vector $h\in\R^d$   a {\sl tilt}  and   elements $\xi\in\Uset$    {\sl directions} or  {\sl velocities}.
Let us  assume   $\gpl^\beta(0)$    finite.  
Then for $\xi\in\ri\Uset$, 
the a.s.\ point-to-point limits \eqref{p2pV_0}--\eqref{p2pV_0LPP} define nonrandom, bounded,  concave,
continuous
functions  $\gpp^\beta:\ri\Uset\to\R$ for  $\beta\in(0,\infty]$   (see Theorem 2.4 and 2.6 and Remark 2.5 of \cite{Ras-Sep-14}).   
  The results of this section do not touch the
relative boundary of $\Uset$.  Consequently we do not need additional 
 assumptions that  guarantee regularity  of $\gpp^\beta$ up to the
 boundary.     One sufficient  assumption would be  the directed i.i.d.\ $L^{d+\e}$ of Remark \ref{rmk:dir-iid} (Theorem 3.2 of \cite{Ras-Sep-14}). 

\begin{remark}  To illustrate what can go wrong on the boundary of $\Uset$, suppose $z\in\range$ is an extreme point of $\Uset$.  Then the only path from $0$ to $nz$ is $x_k=kz$, and   we get  $\gpp^\beta(z)=\beta^{-1}\log p(z) + \E[\Vw(\w, z)\,\vert\,\cI_z]$ where $\cI_z$ is the $\sigma$-algebra of events invariant under the mapping $T_z$.   This can be random even if $\P$ is assumed ergodic under the full  group $\{T_x\}$.    In general  $\gpp^\beta$ is lower semicontinuous on all of $\Uset$, for a.e.\ fixed $\w$ (Theorem 2.6 of \cite{Ras-Sep-14}). 
    \hfill $\triangle$
 
 \end{remark}

With definitions  \eqref{p2lh-beta}--\eqref{p2lh-lim}  and  \eqref{p2pV_0}--\eqref{p2pV_0LPP},    equation  \eqref{eq:sup p2p:lpp}   becomes
\begin{align}\label{eq:velocity-tilt}
\gpl^\beta(h)=\sup_{\xi\in\Uset}\bigl\{\gpp^\beta(\xi)+h\cdot\xi\bigr\}, \qquad h\in\R^d.
\end{align}
In order to invert this relationship between $\gpl^\beta$ and $\gpp^\beta$ we turn it into a convex (or rather, concave)  duality.   First extend $\gpp^\beta$ outside $\Uset$ via  $\gpp^\beta(\xi)=-\infty$ for
$\xi\in\Uset^c$, and   then replace $\gpp^\beta$ with its upper semicontinuous
regularization  
$\bargpp^\beta(\xi)=  \gpp^\beta(\xi)\vee \varlimsup_{\zeta\to\xi}\gpp^\beta(\zeta) $.  
Now \eqref{eq:velocity-tilt}  extends to 
\[    \gpl^\beta(h)=\sup_{\xi\in\R^d}\{\bargpp^\beta(\xi)+h\cdot\xi\}, \qquad h\in\R^d,  \]
which   standard convex duality  \cite{rock-ca} inverts  to  
\[ \bargpp^\beta(\xi)=\inf_{h\in\R^d}\{\gpl^\beta(h)-h\cdot\xi\}, \qquad \xi\in\R^d. \]
 By the continuity of $\gpp^\beta$ on 
   $\ri\Uset$,  the last display  gives 
\begin{align}\label{eq:tilt-velocity}
\gpp^\beta(\xi)=\inf_{h\in\R^d}\big\{\gpl^\beta(h)-h\cdot\xi\big\} \quad \text{ for $\xi\in\ri\Uset$.}
\end{align}

Equations   \eqref{eq:velocity-tilt} and \eqref{eq:tilt-velocity} suggest the next definition, and then Lemma \ref{lm:zeta->h} answers part of the natural next question.  



 \begin{definition}\label{def:h-zeta}
At a fixed $\beta\in(0,\infty]$, we say that  tilt $h\in\R^d$  and  velocity $\xi\in\ri\Uset$ are {\rm dual} to each other if 
 \begin{align}  \gpl^\beta(h)=\gpp^\beta(\xi)+h\cdot\xi.\label{h-zeta}\end{align}
 \end{definition}
 

 \begin{lemma}\label{lm:zeta->h}   Fix $\beta\in(0,\infty]$.
Assume $\P$ ergodic,  $\Vw\in\cL$ and $\gpl^\beta(0)<\infty$. 
Then every $\xi\in\ri\Uset$ has a dual $h\in\R^d$. 
Furthermore, if $h$ is dual to $\xi\in\Uset$ and $h'$ is such that 
\begin{align}
(h-h')\cdot(z-z')=0\text{ for all $z,z'\in\range$}\label{eq:h-unique}
\end{align}
then $h'$ is also dual to $\xi$.
\end{lemma}


\begin{proof}
We start with the proof of the second claim. If \eqref{eq:h-unique} holds then directly from \eqref{p2lh-beta}--\eqref{p2lh-lim},  $\gpl^\beta(h')=\gpl^\beta(h)+(h'-h)\cdot z$ for all $z\in\range$. Hence, $\gpl^\beta(h')-h'\cdot\xi=\gpl^\beta(h)-h\cdot\xi$ and 
$h$ is dual to $\xi$ if and only if $h'$ is.

The equality above also implies that any $h$ in \eqref{eq:tilt-velocity} can be replaced by any $h'$ satisfying \eqref{eq:h-unique}.
Fix  $z_0\in\range$.     One way to satisfy  \eqref{eq:h-unique} is to let  
  $h'$ be  the  orthogonal projection of $h$  onto the linear span $\cV$ of $\range-z_0$.  
 Consequently we can restrict the infimum in \eqref{eq:tilt-velocity} to  
  $h\in\cV$. (This can be all of $\R^d$.)

For any $z\in\range$, $h\in\R^d$, and $\beta\in(0,\infty]$,   
\[\gpl^\beta(h)\ge\E[\Vw(\w,z)]+h\cdot z +\beta^{-1}\log p(z).\]
To see this, for $z\ne 0$ consider the path $x_k=kz$ and use the ergodic theorem. For 
$z=0$ consider a path that finds $\Vw(T_x\w, 0)$ within $\e$ of $\esssup\Vw(\cdot\,,0)$ and 
stays there.

Furthermore, \eqref{eq:sup p2p:lpp} gives $\gpp^\beta(\xi)\le \gpl^\beta(0).$
Consequently we can restrict the infimum in \eqref{eq:tilt-velocity} to $h\in\cV$   that satisfy
\begin{align*}
\h\cdot(z-\xi)\le \gpl^\beta(0)+1-\E[\Vw(\w,z)] -\beta^{-1}\log p(z)  \, \le\, c 
\end{align*}
for all $z\in\range$ and a constant $c$.
Convex combinations  over $z$ lead to 
$\h\cdot(\eta-\xi)\le c$ for all $\eta\in\Uset$.  
  By the definition of relative interior,    $\xi\in\ri\Uset$  implies that for some $\e>0$, $\zeta\in\Uset$ for all $\zeta\in\aff\Uset$   such that $\abs{\xi-\zeta}\le\e$. 
Since $h\in\cV$,  
$\eta=\xi+\e\abs{h}^{-1} h$ lies in $\aff\Uset$ and then by choice of $\e$ also in $\Uset$.  We conclude that   $\e\abs{h}\le c$ and thereby 
  that
 the infimum in \eqref{eq:tilt-velocity} can be restricted to a compact set. 
 Continuity of $\gpl^\beta$ implies that the infimum 
  is achieved and existence of an $h$ dual to $\xi$ has been established.\hfill\hfill\qed
 %
\end{proof}


With these preliminaries  we  extend Theorem \ref{th:K-var}  to the point-to-point case.   Recall Definition \ref{def:cK} of the space $\cK$ of stationary $L^1$ cocycles.  

\begin{theorem}\label{th:K-var-pp}
Assume  $\Vw\in\cL$, $\P$ ergodic and $\gpl^\beta(0)$ finite. Then we have these variational formulas for $\xi\in\ri\Uset$.  
\be	\gpp^\beta(\xi)\;=\;\inf_{B\in\cK} \,\P\text{-}\esssup_\w\;  \beta^{-1}\log \sum_{z\in\range} p(z)e^{\beta \Vw(\w,z)-\beta B(\w, 0,z)-\beta h(B)\cdot \xi} \label{eq:K-var-pp}  \ee
for $0<\beta<\infty$ and 
	\begin{align}
	\gpp^\infty(\xi)&\;=\;\inf_{B\in\cK}\, \P\text{-}\esssup_\w\;  \max_{z\in\range}\, \{\Vw(\w,z)-   B(\w, 0,z)-  h(B)\cdot \xi\}.  \label{eq:g:K-var-pp}
	\end{align}
The infimum in \eqref{eq:K-var-pp}--\eqref{eq:g:K-var-pp} can be restricted to $B\in\cK$  such that $h(B)$ is dual to $\xi$.
For each $\xi\in\ri\Uset$ and 	$0<\beta\le\infty$, there exists   a minimizing $B\in\cK$  such that $h(B)$ is dual to $\xi$.   
 \end{theorem}

\begin{proof}   We write the proof for $0<\beta<\infty$, the case $\beta=\infty$ 
being similar enough.     
The right-hand side of \eqref{eq:K-var-pp} equals 
\begin{align*}
&\inf_h \Bigl\{ \, \inf_{B: h(B)=h} \P\text{-}\esssup_\w\;  \beta^{-1}\log \sum_{z\in\range} p(z)e^{\beta \Vw(\w,z)-\beta B(\w, 0,z)}   -  h\cdot \xi \Bigr\} \\
&=\inf_h \Bigl\{ \, \inf_{F\in\cK_0} \P\text{-}\esssup_\w\;  \beta^{-1}\log \sum_{z\in\range} p(z)e^{\beta \Vw(\w,z)+\beta h\cdot  z +\beta F(\w, 0,z)}   -  h\cdot \xi \Bigr\} \\
&=\inf_h \{ \, \gpl^\beta(h) -  h\cdot \xi\} = \gpp^\beta(\xi).  
\end{align*}
The middle equality is true because $B$ is a cocycle with $h(B)=h$ if and only if 
$F(\w,0,z)=- B(\w, 0,z)-h\cdot z$ is a centered cocycle.  

For the existence, use   Lemma \ref{lm:zeta->h} to pick $h$ dual to $\xi$, and then 
Theorem \ref{th:K-var} to find a minimizing $F\in\cK_0$ for $\gpl^\beta(h)$.  
Then $B(\w,0,z)=-h\cdot z-F(\w,0,z)$ is a minimizer for $\gpp^\beta(\xi)$ and $h(B)=h$. \hfill\hfill\qed
\end{proof}

  Combining Theorems  \ref{thm:minimizer} and \ref{th:K-var-pp} with \eqref{h-zeta}  gives: 
  
 \begin{corollary} \label{cor:h-xi}  Assume  $\Vw\in\cL$, $\P$ ergodic and $\gpl^\beta(0)$ finite.  Let $\beta\in(0,\infty]$ and  $\xi\in\ri\Uset$.  Suppose there exists $B\in\cK$ adapted to $V_0$ {\rm(}Definition \ref{def:bdry-model}{\rm)} and such that $h(B)$ is dual to $\xi$.    Then $B$ minimizes  in \eqref{eq:K-var-pp} or \eqref{eq:g:K-var-pp} without the essential supremum over $\w$ and 
  \begin{align}\label{eq:p2p=B.xi}
\gpp^\beta(\xi)=-h(B)\cdot\xi . 
\end{align}
 \end{corollary}   

 If $\nabla\gpp^\beta$ exists at $\xi$,   the duality of $h(B)$ and $\xi$ implies  that 
 \begin{align}\label{eq:grad g=-h}
\nabla\gpp^\beta(\xi)=-h(B).
\end{align}
  In some situations $\Uset$ has  empty interior but $\gpp^\beta$ extends as a homogeneous function to an open neighborhood of $\Uset$, and \eqref{eq:grad g=-h} makes sense for the extended function.    Such is the case for example  when $\range=\{e_1,\dotsc,e_d\}$.   In  the 1+1 dimensional exactly solvable models discussed 
 in Section \ref{sec:lg+exp}  below,   for each $\xi\in\ri\Uset$ there exists   a cocycle $B=B^\xi$ that satisfies  \eqref{eq:p2p=B.xi} and \eqref {eq:grad g=-h}.   Modulo some regularity issues,  this is the case also for  the 1+1 dimensional 
 corner growth model with general weights   \cite{geor-rass-sepp-lppbuse}.

\section{Cocycles from Busemann functions}  \label{sec:bus} 

The solution    approach advanced in  this paper for the cocycle variational formulas 
 relies  on cocycles that are adapted to 
$\Vw$ (Definition \ref{def:bdry-model}).  
This section   describes  how to obtain  such cocycles    from limits of gradients of free energy, called 
  {\sl Busemann functions}, provided such limits exist.    Busemann functions  come in two variants,  point-to-point and point-to-level.  These are  treated in  the next two theorems.    Proofs of the theorems are at the end of the section.  

\smallskip 
  
We assume now that  every admissible path between two given
points  $x$ and $y$ has the same number of steps.  This prevents loops.  The natural examples are
$\range=\{e_1,e_2,\dotsc, e_d\}$  and   
$\range=\{ (z',1):  z^\prime\in\range^\prime\}$ for some finite $\range^\prime\subset\Z^{d-1}$.    For $x,y\in\Z^d$ such that $y$ can be reached from $x$ 
  define the free energy  
\be\label{Gxy5}  \Gpp^\beta_{x,y} \;= \; 
\beta^{-1}\log\sum_{\substack{n\ge1\\x_{0,n}:\,x_0=x,\, x_n=y}}p(x_{0,n})\,e^{\beta\sum_{k=0}^{n-1}\Vw(T_{x_k}\w, \,z_{k+1})}  \quad\text{for $0<\beta<\infty$}\ee
and the last-passage time 
\be\label{Gxy6} \Gpp^\infty_{x,y} \;= \;\max_{\substack{n\ge1\\x_{0,n}: \,x_0=x,\, x_n=y}}\;\;\sum_{k=0}^{n-1}\Vw(T_{x_k}\w, z_{k+1}).
 \ee
The sum and the 
maximum are taken  over all admissible paths from $x$ to $y$, and then there is a unique  $n$, namely the number of steps   from $x$ to $y$.    

Recall definition \eqref{eq:def:xhat} of the path $\xhat_n(\xi)$. 
   A point-to-point    Busemann function in direction $\xi\in\ri\Uset$ is    defined by 
 \be \Bpp^\xi(x,y)=\lim_{n\to\infty}\bigl[ \Gpp^\beta_{x,\,\xhat_{n}(\xi)+z}-\Gpp^\beta_{y,\,\xhat_{n}(\xi)+z}\,\bigr] ,  \quad x,y\in\cG, \ z\in\range\cup\{0\}, 
\label{bus3}\ee
provided that the limit exists $\P$-almost surely  and does not depend on $z$.     The extra perturbation by $z$ on the right-hand side  will be used to  establish stationarity of the limit.  
$\beta$ is now fixed and we omit  the dependence of $\Bpp^\xi$ on $\beta$ from the notation.   
To ensure that  paths to $\xhat_n(\xi)$  from both $x$ and $y$   exist in \eqref{bus3},  in the definition  \eqref{eq:def:xhat}
 of $\xhat_n(\xi)$  pick   $\alpha_z(\xi)>0$ for all $z\in\range$.  (For $\xi\in\ri\Uset$ this is possible by 
Theorem 6.4 in \cite{rock-ca}.) 
Then, any point $x\in\gr$ can reach $\xhat_n(\xi)$ with steps in $\range$   for  large enough $n$.  




\begin{theorem}\label{th:Bus=grad(a)}
Let $\beta\in(0,\infty]$, $\Vw\in\cL$, $\P$ ergodic and $\gpl^\beta(0)$ finite.  
 Assume that every admissible path between two given
points  $x$ and $y$ has the same number of steps.  

 Fix 
  $\xi\in\ri\Uset$  and choose   $\alpha_z(\xi)>0$ for each  $z\in\range$ in \eqref{eq:def:xhat}.  
Assume that  for all $x,y\in\gr$ and $\P$-a.e.\ $\w$,  the limits \eqref{bus3}
 exist for $z\in\range\cup\{0\}$ and   are independent of $z$.  
Then $\Bpp^\xi(x,y)$ is a stationary cocycle that is adapted to 
$\Vw$ in the sense of Definition \ref{def:bdry-model}.   

Assume additionally 
	\begin{align}\label{Gppliminf}
		\varlimsup_{n\to\infty} n^{-1}\E[\Gpp^\beta_{0,\,\xhat_n(\xi)}]\le\gpp^\beta(\xi).
	\end{align}
Then $\Bpp^\xi(x,y)\in L^1(\P)$ $\forall x,y\in\gr$, $h(\Bpp^\xi)$ is dual to $\xi$ {\rm(}Definition \ref{def:h-zeta}{\rm)}, and  $\gpp^\beta(\xi)=-h(\Bpp^\xi)\cdot\xi$.  \end{theorem}


The point  of  the theorem is that  the Busemann function furnishes correctors for the variational formulas.  Once the assumptions of Theorem \ref{th:Bus=grad(a)}  are satisfied,   (i)   Theorem \ref{thm:minimizer}  implies that 
$F(x,y)=h(\Bpp^\xi)\cdot (x-y)-  \Bpp^\xi(x,y)$ is a corrector  for $\gpl^\beta(h)$ for any $h$ such that $h-h(\Bpp^\xi)\perp\aff\range$, and (ii) 
depending on $\beta$,  $\Bpp^\xi$ minimizes either \eqref{eq:K-var-pp} or \eqref{eq:g:K-var-pp} without the $\P$-essential supremum. 


\medskip 

In the  point-to-level  case  the free energy and last-passage time 
  for paths of length $n$ started at $x$ are defined by a shift 
  $\Gpl^\beta_{x,(n)}(h)(\w)=\Gpl^\beta_{0,(n)}(h)(T_x\w)$. 
 Point-to-level   Busemann functions are defined by 
   \be \Bpl^h(0,z)=\lim_{n\to\infty}\bigl[\Gpl^\beta_{0,(n)}(h)-\Gpl^\beta_{z,(n-1)}(h)\bigr], \qquad z\in\range, 
\label{buse4}\ee
omitting again the $\beta$-dependence from the notation.

\begin{theorem}\label{th:Bus=grad(b)}   Let $\beta\in(0,\infty]$, $\Vw\in\cL$, $\P$ ergodic and $\gpl^\beta(0)$ finite.     Assume that every admissible path between any two given
points  $x$ and $y$ has the same number of steps. 

 Fix  $h\in\R^d$.  Assume  the $\P$-a.s.\  limits \eqref{buse4} 
 exist for all $z\in\range$.  
Then we can extend $\{\Bpl^h(0,z)\}_{z\in\range}$ to a stationary 
cocycle $\{\Bpl^h(x,y)\}_{x,y\in\gr}$, and 
cocycle  $\Bpl^h(x,y)-h\cdot(y-x)$   is adapted to 
$\Vw$  in the sense of Definition \ref{def:bdry-model}. 

Assume additionally 
	\begin{align}\label{Gplliminf}
		\varliminf_{n\to\infty} n^{-1}\E[\Gpl^\beta_{0,(n)}(h)]\le\gpl^\beta(h). 
	\end{align}
Then  $\Bpl^h(x,y)\in L^1(\P)$ for $x,y\in\gr$.   $F(\w,x,y)=h(\Bpl^h)\cdot (x-y)-  \Bpl^h(\w,x,y)$ is a minimizer in  \eqref{eq:Lambda:K-var} for $\gpl^\beta(h)$  if  $0<\beta<\infty$ and in  \eqref{eq:g:K-var} if  $\beta=\infty$. 
\end{theorem}
 
Remark \ref{pl-cor}  below indicates how the  theorem could be upgraded   to state that the minimizer  $F$ is also a corrector, in other words satisfies  \eqref{eq:Kvar:minbeta}  or \eqref{eq:Kvar:min}.

 \begin{remark} 
Assumptions \eqref{Gppliminf} and \eqref{Gplliminf} need to be verified separately for the case at hand.   
In the directed i.i.d.\ $L^{d+\e}$ case of Remark  \ref{rmk:dir-iid}, we can use 
lattice animal bounds:   Lemma 3 from page 85 of \cite{Gand-Kest-94}  gives  
 $  \sup_n   \E\bigl[ \bigl( n^{-1} \Gpp^\beta_{0,\,\xhat_{n}(\xi)}\bigr)^2\,\bigr]  < \infty$
 and $\sup_n   \E\bigl[ \bigl( n^{-1} \Gpl^\beta_{0,(n)}(h)\bigr)^2\,\bigr] < \infty$,  
 which imply  
$L^1$ convergence in  \eqref{p2pV_0}--\eqref{p2pV_0LPP} and \eqref{p2lh-lim}, respectively.  A completely general  sufficient condition is to have  $\Vw$  bounded above.       \hfill $\triangle$
\end{remark}



\begin{remark}   All of the assumptions and conclusions of  Theorems \ref{th:Bus=grad(a)}--\ref{th:Bus=grad(b)} can be  verified   in the exactly solvable cases.   
   In the explicitly solvable 1+1 dimensional cases the   Busemann  limits 
   $\Bpp^\xi$ and $\Bpl^h$ are connected by the duality of $\xi$ and $h$, and  lead to  the same set of cocycles, as described in the next section.  This also holds for the general  1+1 dimensional corner growth model under local regularity assumptions on the shape that ensure the existence of Busemann functions \cite{geor-rass-sepp-lppbuse}.   We would expect this feature  to be true very  generally.  
   \hfill $\triangle$  \end{remark}  
  
  \begin{remark} 
According to 
 \eqref{bus3},    $\Bpp^\xi$ is a microscopic gradient of free energy and passage times in direction $\xi$, and by   \eqref{eq:grad g=-h}   its average gives the macroscopic gradient.  
This form of \eqref{eq:grad g=-h} was anticipated in \cite{How-New-01} in the context of Euclidean first passage percolation (FPP), where $\gpp(x,y)=c\sqrt{x^2+y^2}$ for some $c>0$. 
(See the paragraph after the proof of Theorem 1.13 in \cite{How-New-01}.) 
A version of the formula also appears in Theorem 3.5 of \cite{Dam-Han-14} in the context of nearest-neighbor FPP.  
   \hfill $\triangle$ \end{remark} 

\begin{example}\label{ex:weak1.01}({\it Directed polymer in weak disorder}) 
The directed polymer  in weak disorder illustrates   Theorem \ref{th:Bus=grad(b)}.  We continue with the notation from Example \ref{ex:weak1} and take $\beta>0$ small enough.   Then $\P$-almost surely for $z\in\range$,    
\begin{align*}
&\Gpl^\beta_{0,(n)}(h)-\Gpl^\beta_{z,(n-1)}(h) \\
&\qquad= 
\beta^{-1}\log  W_n  - \beta^{-1}\log  W_{n-1}\circ T_z  
+   \beta^{-1}(\lambda(\beta)+\kappa(\beta h) ) \\
&\qquad\underset{n\to\infty}\longrightarrow  \; \; 
\beta^{-1}\log  W_\infty  - \beta^{-1}\log  W_{\infty}\circ T_z  
+   \beta^{-1}(\lambda(\beta)+\kappa(\beta h) ) \\
&\qquad=  -F(0,z)+  \gpl^\beta(h), 
\end{align*}
with $F$ defined by \eqref{weak:F}.  
Thus the Busemann function is  $\Bpl^h(0,z)=-  F(0,z)+\gpl^\beta(h)$.  By Theorem  \ref{th:Bus=grad(b)},  cocycle  $\Bpl^h(0,z)-h\cdot z$ is adapted to 
$V_0$, as already observed in Example \ref{ex:weak1}.   The Busemann function recovers the corrector $F$ identified  in Example \ref{ex:weak1}. 
\hfill $\triangle$
\end{example}

In the remainder of the section we prove Theorems \ref{th:Bus=grad(a)} and \ref{th:Bus=grad(b)} and then comment on getting a corrector in Theorem \ref{th:Bus=grad(b)}.

\begin{proofof}{of Theorem \ref{th:Bus=grad(a)}}
    To check stationarity, for $z\in\range$ 
\begin{align*}
\Bpp^\xi(z+x,z+y)&=\lim_{n\to\infty}[ \Gpp^\beta_{z+x,\,z+\xhat_{n}(\xi)}-\Gpp^\beta_{z+y,\,z+\xhat_{n}(\xi)}]\\
&=\lim_{n\to\infty}[\Gpp^\beta_{x,\,\xhat_{n}(\xi)}-\Gpp^\beta_{y,\,\xhat_{n}(\xi)}]\circ T_z  =\Bpp^\xi(x,y)\circ T_z.\end{align*}
 Additivity is satisfied by telescoping sums. 
 The condition of Definition \ref{def:bdry-model} is readily checked.  For example,
 in the $\beta=\infty$ case, if  $x$ is reachable   from  $0$ and from  every $z\in\range$,
$\max_{z\in\range}\{\Vw(\w,z)+\Gpp^\infty_{z,x}-\Gpp^\infty_{0,x}\}=0$ because some $z\in\range$ is the first step of a maximizing path from $0$ to $x$.  

Assume \eqref{Gppliminf}.  Recall \eqref{V*}.    Fix $\ell\in\N$ large enough so that, for each $k\ge m\ge 1$, there exists an admissible path $\{y_i^{m,k}\}_{i=0}^{\ell}$ from $\xhat_{k-m}(\xi)$ to  $\xhat_{k+\ell}(\xi)-\xhat_{m}(\xi)$. 
Then 
\be\label{bus-77}
\Gpp^\beta_{0,\, \xhat_{k+\ell}(\xi)-\xhat_{m}(\xi)}(\w) \ \ge \   \Gpp^\beta_{0,\,  \xhat_{k-m}(\xi)}(\w)  +\beta^{-1}\log p(y^{m,k}_{0,\ell}) - \sum_{i=0}^{\ell-1} \Vw^*(T_{y_i^{m,k}}\w) . 
\ee
 By \eqref{bus-77},  for $0<m<n$, 
\begin{align*}
&\frac1{(m+\ell)n} \sum_{k=m}^n \E\bigl[   \Gpp^\beta_{0,\, \xhat_{k+\ell}(\xi)} - \Gpp^\beta_{\xhat_{m}(\xi),\, \xhat_{k+\ell}(\xi)} \, \bigr] \\
&\qquad = \frac1{(m+\ell)n} \sum_{k=m}^n \E\bigl[   \Gpp^\beta_{0,\, \xhat_{k+\ell}(\xi)} - \Gpp^\beta_{0,\, \xhat_{k+\ell}(\xi)-\xhat_{m}(\xi)} \, \bigr]  \\
&\qquad
\le   \frac1{(m+\ell)n} \sum_{k=m}^n \E\bigl[   \Gpp^\beta_{0,\, \xhat_{k+\ell}(\xi)} - \Gpp^\beta_{0,\, \xhat_{k-m}(\xi)} \, \bigr] 
\; + \; \frac{\log p(y^{m,k}_{0,\ell})}{\beta (m+\ell)} \; + \; \frac{\ell\E(\Vw^*)}{m+\ell}\\
&\qquad
\le   \frac1{(m+\ell)n} \sum_{k=n-m+1}^{n+\ell}  \E[   \Gpp^\beta_{0,\, \xhat_{k}(\xi)} ] 
\;-\; 
 \frac1{(m+\ell)n} \sum_{k=0}^{m+\ell-1}  \E[   \Gpp^\beta_{0,\, \xhat_{k}(\xi)} ] 
\; + \; \frac{C}{m}  
\end{align*}
where the last $C$ depends on the fixed $\ell$.
By \eqref{Gppliminf} we get the upper bound 
	\begin{align}
	\label{upper007}
	\varliminf_{n\to\infty} \frac1{(m+\ell)n} \sum_{k=m}^n \E\bigl[   \Gpp^\beta_{0,\, \xhat_{k+\ell}(\xi)} - \Gpp^\beta_{\xhat_{m}(\xi),\, \xhat_{k+\ell}(\xi)} \, \bigr] \le \gpp^\beta(\xi)+\frac{C}m.
	\end{align}
On the other hand, by superadditivity,  
	\[\Gpp^\beta_{0,\, \xhat_{k+\ell}(\xi)} - \Gpp^\beta_{\xhat_{m}(\xi),\, \xhat_{k+\ell}(\xi)}\ge\Gpp^\beta_{0,\,\xhat_m(\xi)}\]
and hence $\frac1{(m+\ell)n} \bigl[ \sum_{k=m}^n  \bigl(  \Gpp^\beta_{0,\, \xhat_{k+\ell}(\xi)} - \Gpp^\beta_{\xhat_{m}(\xi),\, \xhat_{k+\ell}(\xi)} \, \bigr)\bigr]^-$ is uniformly integrable as $n\to\infty$.   Since by assumption \eqref{bus3} 
\[   \frac1{m+\ell} \, \Bpp^\xi(0,\xhat_m(\xi)) 
=   \lim_{n\to\infty} \frac1{(m+\ell)n} \sum_{k=m}^n \bigl[   \Gpp^\beta_{0,\, \xhat_{k+\ell}(\xi)} - \Gpp^\beta_{\xhat_{m}(\xi),\, \xhat_{k+\ell}(\xi)} \, \bigr]
\quad\text{$\P$-a.s.}  \]
we can apply Lemma \ref{app:lm1} from the appendix  to conclude that  
 $\Bpp^\xi(0,\xhat_m(\xi))$ is integrable   and satisfies
 \begin{align}\label{bus8} 
 \frac1{m+\ell} \,\E [\Bpp^\xi(0,\xhat_m(\xi))]  \le  \gpp^\beta(\xi)  \; + \; \frac{C}{m}.
\end{align}

Now we can show 
 $\Bpp^\xi(0,z)\in L^1(\P)$     $\forall z\in\range$.   We have assumed that each step $z$ appears along the path $\xhat_m(\xi)$, so it suffices to observe that  
 \begin{align*} 
 & \Bpp^\xi(0, \xhat_m(\xi)-\xhat_{m-1}(\xi)) \circ T_{\xhat_{m-1}(\xi)}  \\
&\qquad\qquad\qquad\qquad
=  \Bpp^\xi(0, \xhat_m(\xi)) - \Bpp^\xi(0, \xhat_{m-1}(\xi))
\; \in \; L^1(\P). \end{align*}  

We have established that   $\Bpp^\xi$ is a stationary $L^1(\P)$ cocycle that is adapted to 
$\Vw$ in the sense of Definition \ref{def:bdry-model}. 
By definition \eqref{EB}, the left-hand side of \eqref{bus8}   equals 
\[  -(m+\ell)^{-1} h(\Bpp^\xi)\cdot \xhat_m(\xi) \; \to \; - h(\Bpp^\xi)\cdot  \xi
\qquad\text{as $m\to\infty$.} \]  
We have     $- h(\Bpp^\xi)\cdot  \xi\le  \gpp^\beta(\xi) $. 
Since $\gpl^\beta(h(\Bpp^\xi))=0$ by  Theorem \ref{thm:minimizer},  variational formula \eqref{eq:tilt-velocity}  gives the opposite inequality  $- h(\Bpp^\xi)\cdot  \xi\ge  \gpp^\beta(\xi) $.  Duality of 
$h(\Bpp^\xi)$ and  $ \xi$ has been established. \hfill\hfill\qed
\end{proofof}

\begin{proofof}{of Theorem \ref{th:Bus=grad(b)}}
We  check that limits \eqref{buse4} define a stationary cocycle  $\Bpl^h(\w,x,y)$. Fix $x,y\in\gr$ such that there is a  path $x_{0,\ell}$ with increments $z_i=x_i-x_{i-1}\in\range$ that goes from $x=x_0$ to $y=x_\ell$.  By shifting the $n$-index, 
\be\label{bus15} \begin{aligned}
\sum_{i=0}^{\ell-1} \Bpl^h(T_{x_i}\w, 0, z_{i+1})
&= \lim_{n\to\infty}  \sum_{i=0}^{\ell-1} [\Gpl^\beta_{x_i,(n)}(h)-\Gpl^\beta_{x_{i+1},(n-1)}(h)] \\
&= \lim_{n\to\infty}  \sum_{i=0}^{\ell-1} [\Gpl^\beta_{x_i,(n-i)}(h)-\Gpl^\beta_{x_{i+1},(n-i-1)}(h)]\\  
&= \lim_{n\to\infty}    [\Gpl^\beta_{x_0,(n)}(h)-\Gpl^\beta_{x_{\ell},(n-\ell)}(h)]\\
&= \lim_{n\to\infty}    [\Gpl^\beta_{0,(n)}(h)-\Gpl^\beta_{y-x,(n-\ell)}(h)]\circ T_{x}. 
\end{aligned}\ee  
By assumption each path from $x$ to $y$  has the same number $\ell$ of steps. 
Hence we can define $\Bpl^h(\w,x,y)=\sum_{i=0}^{\ell-1} \Bpl^h(T_{x_i}\w,0, z_{i+1})$ independently of the particular steps $z_i$ taken, and with the property  
$\Bpl^h(\w,x,y)=\Bpl^h(T_x\w,0,y-x)$.  

 If $y$ is not accessible from $x$, pick  a point $\bar x$ from which both $x$ and $y$ are accessible and set  $\Bpl^h(\w,x,y)=\Bpl^h(\w,\bar x,y)-\Bpl^h(\w,\bar x,x)$.  This definition is independent of the point $\bar x$.    Now we have a  stationary cocyle $\Bpl^h$.  


 A   first step decomposition of $\Gpl^\beta_{0,(n)}(h)$ shows that cocycle  
 \be\label{bus13} \wt B(0,z)=\Bpl^h(0,z)-h\cdot z \ee
satisfies Definition \ref{def:bdry-model}.  
   
 Under assumption    \eqref{Gplliminf}  the integrability of  $\Bpl^h(0,z)$ is proved exactly as in the proof of Theorem \ref{th:Bus=grad(a)}.  First an upper bound:   
	\begin{align}\label{L1-upper}
	\begin{split}
	&\varliminf_{n\to\infty}n^{-1}\sum_{k=1}^n \E[\Gpl^\beta_{0,(k)}(h)-\Gpl^\beta_{z,(k-1)}(h)]\\
	&\qquad=\varliminf_{n\to\infty}n^{-1}\sum_{k=1}^n \E[\Gpl^\beta_{0,(k)}(h)-\Gpl^\beta_{0,(k-1)}(h)]\\
	&\qquad=\varliminf_{n\to\infty}n^{-1}\E[\Gpl^\beta_{0,(n)}(h)] 
	\le\gpl^\beta(h).
	\end{split}
	\end{align}
 Then uniform integrability of $\bigl[n^{-1}\sum_{k=1}^n (\Gpl^\beta_{0,(k)}(h)-\Gpl^\beta_{z,(k-1)}(h))\bigr]^-$ from the lower bound 
  	\be\label{pl-ui} \Gpl^\beta_{0,(n)}(h)-\Gpl^\beta_{z,(n-1)}(h)\ge \Vw(\w,z)+h\cdot z+\beta^{-1}\log p(z).\ee 
By Lemma \ref{app:lm1},    $\Bpl^h(0,z)\in L^1(\P)$ and  
\be\label{bus17}  -h(\Bpl^h)\cdot z = \E[\Bpl^h(0,z)]\le\gpl^\beta(h)\quad \text{for}\quad z\in\range. \ee  

Define the centered stationary $L^1$ cocycle 
\be\label{bus-F}  F(\w, x,y)=  h(\wt B)\cdot (x-y)-  \wt B (\w, x,y)= 
  h(\Bpl^h)\cdot (x-y)-  \Bpl^h(\w, x,y). \ee 
By variational formula \eqref{eq:Lambda:K-var}, \eqref{bus13}, \eqref{bus17},  and \eqref{VBbeta} applied to $\wt B$, 
\be\label{bus19}\begin{aligned}
&\gpl^\beta(h) \le \P\text{-}\esssup_\w\;  \beta^{-1}\log \sum_{z\in\range} p(z)e^{\beta \Vw(\w,z)+\beta h\cdot z+\beta F(\w, 0,z)}\\
&= \P\text{-}\esssup_\w\;  \beta^{-1}\log \sum_{z\in\range} p(z)e^{\beta \Vw(\w,z)-\beta h(\Bpl^h)\cdot z-\beta \wt B(\w, 0,z)}\\
&\le \gpl^\beta(h) \; + \; \P\text{-}\esssup_\w\;   \beta^{-1}\log \sum_{z\in\range} p(z)e^{ \beta  \Vw(\w,z)  -\beta \wt B(\w, 0,z)} 
= \gpl^\beta(h).  
\end{aligned}\ee  
This shows that $F$ is a minimizer in  \eqref{eq:Lambda:K-var}. A similar proof works for the case $\beta=\infty$.\hfill\hfill\qed
\end{proofof}

\begin{remark}\label{pl-cor}({\it Corrector in Theorem \ref{th:Bus=grad(b)}})
Continue with  the assumptions of Theorem \ref{th:Bus=grad(b)}.  
We point out two sufficient conditions for concluding  that  $F$ of \eqref{bus-F}   is  not merely a minimizing cocycle for  $\gpl^\beta(h)$ as stated in  Theorem \ref{th:Bus=grad(b)},  but also a corrector  for $\gpl^\beta(h)$.  
By Theorem \ref{thm:minimizer},   $F$   is a corrector  for $\gpl^\beta(h')$ for  any $h'$ such that   $h'-h(\wt B)\perp\aff\range$.  
     Since $h'-h(\wt B)= h'-h(\Bpl^h)-h$, 
    for $h'=h$ the condition is  $h(\Bpl^h)\perp\aff\range$, or equivalently that $h(\Bpl^h)\cdot z$ is constant over $z\in\range$.   \eqref{bus17} and \eqref{bus19} (and its analogue for $\beta=\infty$) imply that $-h(\Bpl^h)\cdot z=\gpl^\beta(h)$ for at least one $z\in\range$.   Hence the condition is 
 	\begin{align}\label{-h.z=g}
	-h(\Bpl^h)\cdot z=\gpl^\beta(h)\quad\text{for all }z\in\range.
	\end{align}
Here are two ways to satisfy \eqref{-h.z=g}.  	

{\rm(a)} By the first two equalities in \eqref{L1-upper}, \eqref{-h.z=g} would follow from convergence of expectations  in  \eqref{p2lh-lim} and   Ces\`aro convergence of expectations  in \eqref{buse4}:
\begin{align*}
\E[\Bpl^h(0,z)] &=  \lim_{n\to\infty}\E\Bigl[\,\frac1n\sum_{k=1}^n\bigl(\Gpl^\beta_{0,(k)}(h)-\Gpl^\beta_{z,(k-1)}(h)\bigr)\Bigr]  \\
&= \lim_{n\to\infty}\E\bigl[n^{-1}\Gpl^\beta_{0,(n)}(h)\bigr] = \gpl^\beta(h).  
\end{align*}

%
	
{\rm(b)}   Suppose $h$ is dual to some $\bar\xi\in\ri\Uset$.  Then \eqref{-h.z=g} follows by this argument.   First $\gpl^\beta(h(\Bpl^h)+h)=\gpl^\beta(h(\wt B))=0$  by Theorem \ref{thm:minimizer}.   Then combining \eqref{eq:velocity-tilt} and \eqref{bus17} gives 
	\begin{align}\label{auxauxaux}
	\gpp(\xi)+h\cdot\xi\le-h(\Bpl^h)\cdot\xi\le\gpl^\beta(h) \qquad \forall \xi\in\Uset.  
	\end{align}
From this $-h(\Bpl^h)\cdot\bar\xi=\gpl^\beta(h)$.  Since $\bar\xi\in\ri\Uset$ we can write $\bar\xi=\sum_{z\in\range} \alpha_z z $ where each $\alpha_z>0$,  and now \eqref{bus17} forces \eqref{-h.z=g}.  	
%
%
   \hfill$\triangle$
  \end{remark}

\section{Exactly solvable models in 1+1 dimensions}
\label{sec:lg+exp}

We   describe how the   theory developed  manifests itself in  two well-known  1+1 dimensional  exactly solvable models.  
The setting is the canonical one with  $\Omega=\R^{\Z^2}$,   
$\range=\{e_1,e_2\}$,  $\Uset=\{(s,1-s): 0\le s\le 1\}$, 
and i.i.d.\ weights  $\{\w_x\}_{x\in\Z^2}$ under $\P$.  The   distributions of the weights 
are specified  in the examples below.    


%

\subsection{Log-gamma polymer}  \label{sec:lg} 
	The log-gamma polymer \cite{sepp-12-aop} is an explicitly solvable 
	1+1 dimensional directed polymer model for which the approach of 
  this paper can be  carried out explicitly.  Some details are  in \cite{geor-rass-sepp-yilm-15}. 
	We describe the results briefly.  

	Fix   $0<\rho<\infty$
 	and let   $\w_x$ be Gamma($\rho$)-distributed, i.e.\ 
	$\P\{{\w_x}\le r\}=\Gamma(\rho)^{-1} \int_0^r t^{\rho-1}e^{-t}\,dt$ for $0\le r<\infty$.    
	Inverse temperature is fixed at $\beta=1$. 
	(Parameter $\rho$ can be viewed as  temperature, see Remark 3.2 in 
	\cite{geor-rass-sepp-yilm-15}.)
	The potential is  $\Vw(\w)=-\log \w_0+\log 2$.  
	Let $\Psi_0=\Gamma'/\Gamma$ and $\Psi_1=\Psi_0'$ be
	  the digamma and  trigamma function. 
	  
   Utilizing the stationary version of the log-gamma polymer one can compute the point-to-point  limit  for  $\xi=(s,1-s)$ as
		\be
		\label{eq:lfed}
			\gpp^1(\xi) 
				=\inf_{\theta \in (0,\rho)}\{ -s\Psi_0(\theta)  -(1-s)\Psi_0(\rho-\theta)\}= -s\Psi_0(\theta(\xi)) -(1-s)\Psi_0(\rho-\theta(\xi))
		\ee
where 	$\theta=\theta(\xi)\in(0,\rho)$ is the unique solution of the equation
	 \[ s\Psi_1(\theta)-(1-s)\Psi_1(\rho-\theta)=0.\]    
	 (See Theorem 2.4 in \cite{sepp-12-aop} or Theorem 2.1 in \cite{geor-sepp-13}.) 
From this we solve the tilt-velocity duality   explicitly:  tilt $h=(h_1, h_2)\in\R^2$ and velocity $\xi\in \ri\Uset$ are dual (Definition \ref{def:h-zeta})  if and only if 
		\be
			\label{lg-dual}  
				h_1-h_2 =\Psi_0(\theta(\xi))-\Psi_0(\rho-\theta(\xi)). 
		\ee
Then 
\be\label{lg-gpl} \gpl^1(h)=h_1-\Psi_0(\theta(\xi))=h_2-\Psi_0(\rho-\theta(\xi)). \ee  

	For all $\xi \in \ri\Uset$ and $h\in\R^2$, the point-to-point and point-to-line  Busemann functions  $\Bpp^\xi(\w,0,z)$  and  $\Bpl^\h(\w,0,z)$  exist as the a.s.\ limits  defined by 
	\eqref{bus3} and \eqref{buse4}   (Theorems 4.1 and 6.1 in \cite{geor-rass-sepp-yilm-15}).  They  satisfy 
		\be
			\label{buse6}  
			\Bpl^h(\w, 0,z)=	\Bpp^\xi(\w,0,z)+h\cdot z  \quad\text{for $z\in\range$} 
		\ee
	whenever $\xi$ and $h$ are dual (\cite{geor-rass-sepp-yilm-15}, Theorem 6.1).    All   the assumptions and  conclusions of  Theorems \ref{th:Bus=grad(a)}--\ref{th:Bus=grad(b)}  and Remark \ref{pl-cor} are valid.  
	
The  marginal distributions of the Busemann functions are given by
		\[
			 e^{-\Bpp^\xi(x,x+e_1)}\sim \textrm{Gamma}(\theta(\xi)) 
				\quad \textrm{and} \quad 
			 e^{-\Bpp^\xi(x,x+e_2)}\sim \textrm{Gamma}(\rho-\theta(\xi)).
		\]
Vector 
		\[ 
			 h(\Bpp^\xi)=-\bigl(\,\E[\Bpp^\xi(0,e_1)]\,, \E[\Bpp^\xi(0,e_2)]\,\bigr)=\bigl(\Psi_0(\theta(\xi)), \Psi_0(\rho-\theta(\xi))\bigr)
		\] 
	is dual to $\xi$ and  $\gpp^1(\xi)=-h(\Bpp^\xi)\cdot\xi$   gives   the point-to-point free energy \eqref{eq:lfed}. 	From \eqref{buse6} we deduce  $\E[\Bpl^h(0,z)]=\gpl^1(h)$ for $z\in\{e_1,e_2\}$.

\subsection{Corner growth model with exponential weights}
	\label{sec:exp}
	
	This is   last-passage percolation on $\Z^2$ with admissible steps $\{e_1, e_2\}$ and i.i.d.\ weights $\{ \om_x\} _{x\in \Z^2}$ with rate 1 exponential distribution.  That is,  $\P\{\w_x>s\}=e^{-s}$ for $0\le s<\infty$.    The potential is $V_0(\om)=\om_0$ and then 
	  $G^\infty_{x,y}$ is  as in \eqref{Gxy6}.  This model can be viewed as the
	zero-temperature limit of the log-gamma polymer (Remark 4.3 in 
	\cite{geor-rass-sepp-yilm-15}). 
 	
  Since the   limit shape of  the exponential corner growth model  is known explicitly and has curvature,  Busemann functions can  be derived with the approach of Newman et al.\ by first proving coalescence of geodesics.  This approach was carried out   by Ferrari and Pimentel  \cite{ferr-pime-05}  (see also Sect.~8 of \cite{Cat-Pim-12}).     An alternative approach that begins by constructing stationary cocycles from queueing fixed points  is in   \cite{geor-rass-sepp-lppbuse}.   
  	
		Velocity $\xi = (s,1-s)$ 
	now selects a parameter $\alpha(\xi) =\frac{\sqrt{s}}{ \sqrt{s} + \sqrt{1-s}}\in (0,1)$ that
	characterizes the marginal distributions of the Busemann functions:  
		\[
			\Bpp^{\xi}(x,x+e_1) \sim \textrm{Exp}(\alpha(\xi)) 
				\quad \textrm{ and } \quad 
			\Bpp^{\xi}(x,x+e_2) \sim \textrm{Exp}(1-\alpha(\xi)).
		\] 
  A tilt dual to   $\xi \in \Uset$ is given by
		\[
			h(\xi)=-\bigl(\,\E[\Bpp^\xi(0,e_1)], \E[\Bpp^\xi(0,e_2)]\,\bigr)=-\Bigl(\,\frac{1}{\alpha(\xi)}\,,\, \frac{1}{1- \alpha(\xi)} \,\Bigr).
		\]
	Substituting in \eqref{eq:p2p=B.xi} we obtain  the well-known  limit  formula from  Rost \cite{rost}:
		\[
		\gpp^{\infty}(s,1-s) = 1 + 2\sqrt{s(1-s)}.
		\]

\section{Variational formulas in terms of measures}
\label{sec:entr} 

In this section  we derive  variational formulas for last-passage percolation
 in terms of probability measures on the spaces $\Omega_\ell=\Omega\times\range^\ell$ for 
 $\ell\in\Z_+$.  
This section contains  no new results for positive temperature models, but positive temperature results are recalled and rewritten    for taking a zero-temperature limit.  
 The formulas we obtain are zero-temperature limits
 of polymer  variational formulas   that involve entropy.   A maximizing measure can be identified  for  
polymers in weak (enough) disorder
(Example \ref{ex:weak2} below). 
 In the final Section \ref{sec:finite} we relate these measure variational formulas to Perron-Frobenius theory,  the classical one for $0<\beta<\infty$ and max-plus theory for $\beta=\infty$.   
 
Return now to the setting of  Section \ref{sec:free},  with  general  $\ell\in\Z_+$ and measurable potential 
 $V:\Omega_\ell\to\R$.    For $\beta\in(0,\infty]$ 
 define the point-to-level  and point-to-point  limits 
$\gpl^\beta$ and $\gpp^\beta(\xi)$ by Theorem \ref{th:p2p}.   
A generic element of  $\Omega_\ell$ is denoted by  $\wz=(\w,z_{1,\ell})$ 
with $\w\in\Omega$ and   $z_{1,\ell}=(z_1,\dotsc, z_\ell)\in\range^\ell$.     For  $1\le j\le \ell$ let $Z_j(\w, z_{1,\ell})=z_j$ denote the $j$th  step  variable
on $\Omega_\ell$.  
  On $\Omega_\ell$ introduce the mappings 
\begin{align}\label{eq:Sz}
S_z(\w,z_{1,\ell})=(T_{z_1}\w,(z_{2,\ell-1},z)),\quad z\in\range.
\end{align}
  When $\ell=0$, 
always take   $\Omega_0=\Omega$, 
 $\wz=\w$ and  $S_z=T_z$. 
   In general,  let $b\mathcal X$ denote the space of bounded measurable real-valued  functions on the space $\mathcal X$.

The probability  measures that appear in the variational formula possess  a natural invariance.  This is described in the next proposition, proved at the end of the section.       One manifestation of the invariance will be the following property of a  probability measure $\mu\in\M_1(\Omega_\ell)$ for any $\ell\in\Z_+$:    
\begin{align} 
\label{eq:max>}
E^\mu\big[ \max_{z\in\range} f\circ S_z \big] \ge E^\mu[ f ]
\quad \text{ $\forall\, f\in b\Omega_\ell$.}
\end{align} 
 If $\ell\ge1$ and   $\mu\in\cM_1(\Omega_\ell)$,  let 
$\mu_\ell(\cdot\,\vert\, \w, z_{1,\ell-1})$ denote the conditional 
distribution of  $Z_\ell$ under $\mu$, given  $(\w, z_{1,\ell-1})$.  
We associate to $\mu$  the following  Markov transition  kernel on the space $\Omega_\ell$: 
\be  q_z(\w, z_{1,\ell}) \equiv q\bigl( (\w, z_{1,\ell}), (T_{z_1}\w, (z_{2,\ell},z))\bigr)
=\mu_\ell(z\,\vert\,T_{z_1}\w, z_{2,\ell}) .  \label{q-ker1}\ee
The first notation is a 
convenient abbreviation.  Under this  kernel the state of the  Markov chain on $\Omega_\ell$  jumps from  $(\w, z_{1,\ell})$ to  $(T_{z_1}\w, (z_{2,\ell},z))$
with probability  $\mu_\ell(z\,\vert\,T_{z_1}\w, z_{2,\ell})$ for $z\in\range$. 

Let  $z_{k,\infty}=(z_i)_{k\le i<\infty}$
denote an infinite  sequence of steps indexed by $\{k,k+1,k+2, \dotsc\}$.    It is an element of $\range^{\{k,k+1,k+2, \dotsc\}}$ which we identify with $\range^\N$ in the obvious way.  
 On the space  $\Omega_\N=\Omega\times\range^\N$ define a shift mapping $S$ 
 by  $S(\w,z_{1,\infty})=(T_{z_1}\w,z_{2,\infty})$.  Let $\cM_s(\Omega_\N)$ denote the 
  set of $S$-invariant probability measures on $\Omega_\N$.

\begin{proposition} \label{lm:S}  $  $ \hbox{}  

{\rm Case (a).}    Let $\ell\in\N$ and $\mu\in\cM_1(\Omega_\ell)$.
 Then properties {\rm(a.i)}--{\rm(a.iv)} below are equivalent.  

{\rm (a.i)}   $\mu$ is invariant under kernel   \eqref{q-ker1} defined in terms of $\mu$ itself.   

{\rm (a.ii)}    $\mu$ is the $\Omega_{\ell}$-marginal of an $S$-invariant 
probability measure $\nu\in\cM_s(\Omega_\N)$.

{\rm (a.iii)}   $\mu$  has property \eqref{eq:max>}. 

{\rm (a.iv)}  $\mu$ satisfies this condition: 
 \begin{align} E^\mu[ f(\w, Z_{1,\ell-1}) ] = E^\mu[ f(T_{Z_1}\w, Z_{2,\ell}) ]
\quad \text{ $\forall\, f\in b\Omega_{\ell-1}$.}
\label{Sell-1}\end{align} 

{\rm Case (b).}    Let $\ell=0$ and $\mu\in\cM_1(\Omega)$.
 Then properties {\rm(b.i)}--{\rm(b.iii)} below are equivalent.  

{\rm (b.i)}  There exists a Markov kernel  of the form $\{q_z(\w)\equiv q(\w,T_z\w):z\in\range\}$ on $\Omega$  that fixes $\mu$.  

{\rm (b.ii)}  $\mu$ is the $\Omega$-marginal of an $S$-invariant 
probability measure $\nu\in\cM_s(\Omega_\N)$.

{\rm (b.iii)} $\mu$  has property \eqref{eq:max>} with $S_z=T_z$.  

\end{proposition}

For $\ell\in\Z_+$  let $\M_s(\Omega_\ell)$ denote the space of probability measures described in Proposition \ref{lm:S} above.   
To illustrate, if  $\ell=0$  then $\M_s(\Omega)$ contains all $\{T_x\}$-invariant measures, and if also   $0\in\range$ then $\M_s(\Omega)$ contains  all probability  measures  on $\Omega$.

We can now state the measure variational formulas for point-to-level and point-to-point last-passage percolation limits.   
 For a probability measure $\mu$ on $\Omega_\ell$, $\mu_0$ denotes the 
$\Omega$-marginal: $\mu_0(A)=\mu(A\times\range^\ell)$.   If $\ell=0$ then $\mu_0=\mu$.    $V^-=-\min\{V,0\}$ is the negative part of the function $V$. 

 \begin{theorem}\label{th:g=Hstar}  
Let $\P$ be ergodic, $\ell\in\Z_+$,  and assume $V\in\cL$. 
Then 
	\begin{align} 
	\gpl^\infty&=\sup\Big\{E^\mu[V]:\mu\in\cM_s(\Omega_\ell),\,\mu_0\ll\P,\, E^\mu[V^-]<\infty\Big\}.\label{eq:g:H-var}
	\end{align}
 \end{theorem}
 
 The set in braces in \eqref{eq:g:H-var} is not empty because the measure 
 $\mu(d\w, z_{1,\ell})$ $=$ $\P(d\w)\alpha(z_1)\dotsm\alpha(z_\ell)$ is a member of $\cM_s(\Omega_\ell)$ for any probability $\alpha$ on $\range$ and $V(\cdot\,, z_{1,\ell})\in L^1(\P)$ by the assumption $V\in\cL$.  
 
We state the point-to-point version only for the directed i.i.d.~$L^{d+\e}$ case
defined in  Remark \ref{rmk:dir-iid}.

\begin{theorem}\label{th:ent2:lpp}
Let $\Omega=\cS^{\Z^d}$ be a product of Polish  spaces with  
  shifts $\{T_x\}_{x\in\Z^d}$ and     an i.i.d.\ product measure  $\P$.  
 Let $\ell\in\N$ and assume   $0\notin\Uset$.  Assume that $\forall z_{1,\ell}\in\range^\ell$, 
  $V(\w,z_{1,\ell}) $ is a local function of $\w$ and a member of $L^p(\P)$ for some $p>d$.  
Then  for all  $\xi\in\Uset$, 
\begin{align}
\gpp^\infty(\xi)&=\sup\Big\{E^\mu[V]:\mu\in\cM_s(\Omega_\ell),\,\mu_0\ll\P,E^\mu[V^-]<\infty,E^\mu[Z_1]=\xi\Big\}.\label{eq:g:H-var:p2p}
\end{align}
\end{theorem}

Note that even if $V$ is a function on $\Omega$ only,   variational formula 
\eqref{eq:g:H-var:p2p} uses measures on $\Omega_\ell$ with  $\ell\ge 1$
in order  for the mean step condition 
$E^\mu[Z_1]=\xi$ to make sense.  
 Remark \ref{rmk:tight} below explains  
why Theorem \ref{th:ent2:lpp} is stated only for the directed i.i.d.~$L^{d+\e}$  case. 
In   the general setting  of 
 Theorem \ref{th:g=Hstar} the 
   point-to-point formula \eqref{eq:g:H-var:p2p} is valid 
for  compact $\Omega$
and $\xi\in\ri\Uset$.  
It  can be derived by applying the argument  given below   to the results 
in \cite{Ras-Sep-12-arxiv}.

\medskip
 
 To prepare for the proofs   
we   discuss the   positive temperature setting.   In the end 
  we take $\beta\to\infty$  to prove Theorems \ref{th:g=Hstar}--\ref{th:ent2:lpp}.
  Recall the random walk kernel $p$ from the beginning of Section \ref{sec:free} with ellipticity constant $\delta=\min_{z\in\range}p(z)>0$.  It  acts  as a Markov  transition kernel   on  $\Omega_\ell$ through 
 	\begin{align}\label{eq:def:p}
		&p(\wz,S_z\wz)=p(z) \, \text{ for }z\in\range\text{ and }\wz=(\w,z_{1,\ell})\in\Omega_\ell.    
	\end{align}
This   kernel defines  a  joint Markovian evolution  $(T_{X_n}\w, Z_{n+1,n+\ell})$  of the 
environment  seen by the  $p$-walk $X_n$ and the  vector  
$Z_{n+1,n+\ell}=(Z_{n+1},\dotsc, Z_{n+\ell})$ of the next $\ell$ 
steps   $Z_k=X_k-X_{k-1}$   of the walk.    As before if $\ell=0$ then $S_z=T_z$ and the Markov chain is $T_{X_n}\w$.  


We define an entropy $\Hbar(\mu)$ for probability measures $\mu\in\M_1(\Omega_\ell)$,     associated to this Markov chain and    the
  background measure $\P$.  
 If  $q(\wz,\cdot\,)$ is  a Markov kernel on $\Omega_\ell$
   such that $q(\wz,\cdot\,)\ll p(\wz,\cdot\,)$ $\mu$-a.s., then  $q(\wz,\cdot\,)$ is 
 supported on   $\{S_z\wz\}_{z\in\range}$ and the familiar relative entropy is  
\[H(\mu\times q\,\vert\,\mu\times p)
	= \int_{\Omega_\ell} \sum_{z\in\range}q(\wz,S_z\wz)\,\log\frac{q(\wz,S_z\wz)}{p(\wz,S_z\wz)}\,\mu(d\wz).\]
Set 
 	\begin{align}\label{eq:def:H}
		\Hbar(\mu)=
			\begin{cases}
				\ddd\inf_{q:\,\mu q=\mu}  H(\mu\times q\,|\,\mu\times p) &\text{if }\mu_0\ll\P\\
				\infty&\text{otherwise,}
			\end{cases}
	\end{align}
where the infimum is over Markov kernels $q$ on $\Omega_\ell$ that fix $\mu$, i.e.\ $\mu q(\cdot)\equiv\int q(\wz,\cdot)\mu(d\wz)=\mu(\cdot)$.
 The function  $\Hbar: \cM_1(\Omega_\ell)\to[0,\infty] $ is convex  \cite[Sect.~4]{Ras-Sep-11}.

\begin{remark} 
When $\mu\in\M_s(\Omega_\ell)$ for some $\ell\ge 1$ and $\mu_0\ll\P$,  the minimizing kernel in \eqref{eq:def:H}  is the one  defined in \eqref{q-ker1}, and 
\be\label{H79}   \Hbar(\mu)= H(\mu\,\vert\,\mu_{\ell-1}\otimes p)
= \int_{\Omega_\ell} \,\mu(d\w, dz_{1,\ell}) \log\frac{\mu_\ell(z_\ell\,\vert\,\w,z_{1,\ell-1})}{p(z_\ell)}    
\ee  
where $\mu_{\ell-1}$ is the distribution of $(\w, Z_{1,\ell-1})$ under $\mu$ and $\mu_{\ell-1}\otimes p$ is the product measure on $\Omega_\ell$.  

Here is the argument.  
Let $q(\eta, S_z\eta)=q_z(\eta)$  be  an arbitrary kernel  that fixes $\mu$ and is  supported on   $\{S_z\wz\}_{z\in\range}$.  The first equality below is  the convex dual representation of relative entropy 
(see for example Theorem 5.4 in \cite{Ras-Sep-15-ldp}).   In the second last equality use both $q$-invariance and  \eqref{Sell-1}. 
\begin{align*}  
&H(\mu\times q\,\vert\, \mu\times p)\\
&= 
\sup_{h\in b\Omega_\ell^2} \biggl\{   \sum_z \int\limits_{\Omega_\ell}  h(\eta, S_z\eta)\,q_z(\eta) \,\mu(d\eta)
 -   \log  \sum_z  p(z) \int\limits_{\Omega_\ell}  e^{h(\eta, S_z\eta)} \,\mu(d\eta)
\biggr\} \\
&\ge  \sup_{f\in b\Omega_\ell} \biggl\{   \sum_z \int\limits_{\Omega_\ell}  f( S_z\eta)\,q_z(\eta) \,\mu(d\eta)
-   \log  \sum_z  p(z) \int\limits_{\Omega_\ell}  e^{f( T_{z_1}\w, (z_{2,\ell}, z))} \,\mu(d\w, dz_{1,\ell})
\biggr\} \\
&=  \sup_{f\in b\Omega_\ell} \biggl\{  \;  \int\limits_{\Omega_\ell}  f  \,d\mu 
-   \log  \sum_z  p(z) \!\! \int\limits_{\Omega_{\ell-1}} \!\! e^{f( \w, (z_{1,\ell-1}, z))} \,\mu_{\ell-1}(d\w, d z_{1,\ell-1})
\biggr\}  \\
&= \; H(\mu\,\vert\, \mu_{\ell-1}\otimes p).\tag*{$\triangle$}  
\end{align*} 
\end{remark}

We   state the measure variational formulas for point-to-level and point-to-point polymers in positive temperature.   These are slightly altered versions of   Theorem 2.3 of \cite{Ras-Sep-Yil-13}   and    Theorem 5.3 of \cite{Ras-Sep-14}.  
 
 \begin{theorem}\label{th:var9}  
Let $\P$ be ergodic, $\ell\in\Z_+$,  $0<\beta<\infty$,  and assume $V\in\cL$. 
Then 
	\begin{align} 
	\Lapl^\beta&= \sup\Big\{ E^\mu[V]-\beta^{-1}\Hbar(\mu):\mu\in\cM_s(\Omega_\ell),\,\mu_0\ll\P, \,  E^\mu[ V^- ] < \infty\Big\} .\label{eq:var9.1}
	\end{align}
 \end{theorem}
 
 The quantity inside the  braces cannot be $\infty-\infty$ for the following reason.  By Proposition \ref{lm:S}  every  $\mu\in\M_s(\Omega_\ell)$  is fixed by some  kernel $q$ supported on shifts.  Thereby,  if also  $\mu_0\ll\P$,   the definition of entropy gives 
\be\label{H-9} 0\le\Hbar(\mu)\le \log\ellc^{-1} .    \ee

As above, we state the point-to-point version only for the directed i.i.d.~$L^{d+\e}$ case
defined in  Remark \ref{rmk:dir-iid}.     See Remark \ref{rmk:tight} below for an explanation.

\begin{theorem}\label{th:var10}
  Repeat the assumptions of Theorem \ref{th:ent2:lpp}.  
Then  for  $0<\beta<\infty$ and   $\xi\in\Uset$, 
\begin{align}
\begin{split}
\gpp^\beta(\xi)=\sup\Big\{E^\mu[V]-\beta^{-1}\Hbar(\mu):\,&\mu\in\cM_s(\Omega_\ell),\,\mu_0\ll\P,\\
&E^\mu[V^-]<\infty,E^\mu[Z_1]=\xi\Big\}.
\end{split}\label{eq:var10.1}
\end{align}
\end{theorem}




We illustrate formulas  \eqref{eq:var9.1}  and \eqref{eq:var10.1}
in the case of weak disorder.  

\begin{example}\label{ex:weak2}({\it Directed polymer in weak disorder})
We identify first  the measure $\mu$ that maximizes variational formula 
\eqref{eq:var9.1}  for the directed polymer in weak disorder,
with  potential $V(\w,z)=\Vw(\w)+h\cdot z=\w_0+h\cdot z$ and small enough $0<\beta<\infty$.  
This measure will be invariant for the Markov transition implicitly contained 
in equation \eqref{weak:W}.  
We continue with the notation and assumptions 
from Example \ref{ex:weak1}.  

To define the measure  we need  a backward path and a martingale in the reverse time direction.  
The backward path $(x_k)_{k\le 0}$ satisfies $x_0=0$ and 
$z_k=x_k-x_{k-1}\in\range$, and the corresponding  martingale is 
\[  
W^-_n = e^{-n(\lambda(\beta)+\kappa(\beta h))} \sum_{x_{-n,0}} \abs{\range}^{-n}\,e^{\beta \sum_{k=-n}^{-1}\w_{x_k} - \beta h\,\cdot\, x_{-n}}.
\]  
$W^-_n$  is the same as $W_n$ composed with the reflection $\w_x\mapsto \w_{-x}$, 
and so \eqref{weak:ui} guarantees also $W^-_n\to W^-_\infty$ with the same 
properties.     (Recall that in this example we took  the uniform kernel $p(z)=\abs{\range}^{-1}$.)
  
 By  \eqref{weak:W}  
\[
q^h_0(\w, z)= p(z)\, e^{\beta\w_0-\lambda(\beta)+\beta h\cdot z-\kappa(\beta h)}
\frac{W_\infty(T_z\w)}{W_\infty(\w)}
\]
 defines a stochastic kernel from
$\Omega$ to $\range$.  Define  a Markov transition kernel  on $\Omega\times\range$   by 
\beq q^h((\w, z_{1}), (T_{z_1}\w, z))= q^h_0(T_{z_1}\w, z) .
\label{qtheta}\eeq  
Define  the  probability  
measure $\mu^h$ on   $\Omega\times\range$  as follows. For a bounded Borel  function $\varphi$
\[  \sum_{z\in\range} \int_\Omega \varphi(\w, z)\,\mu^h(d\w, z) = 
\sum_{z\in\range} \int_\Omega  W_\infty^-(\w)\,W_\infty(\w) \,q^h_0(\w, z)\, \varphi(\w, z) \,\P(d\w). \] 
Using the 1-step decomposition of $W_\infty^-$ (analogue of \eqref{weak:W}) one shows  that  
$q^h$ fixes $\mu^h$. 

Let us strengthen assumption \eqref{weak:ui} to also include 
$\E[W_\infty\log^+ W_\infty]<\infty$.  This is true for small enough $\beta$.   Then the entropy can be   calculated: 
\begin{align*}
&H(\mu^h\times q^h\vert\mu^h\times p)\\
&\qquad=\beta E^{\mu^h}[V] -\lambda(\beta)-\kappa(\beta h) 
+\sum_z \int \mu^h_0(d\w)q^h_0(\w, z) \log \frac{W_\infty(T_z\w)}{W_\infty(\w)}\\
&\qquad=\beta E^{\mu^h}[V] -\lambda(\beta)-\kappa(\beta h) 
\end{align*}
because the last term of the middle member vanishes by the invariance. 
$E^{\mu^h}[V] $ is finite  because, by independence and Fatou's lemma,  
\begin{align*} E^{\mu^h}(\abs{\w_0}) =  \E( \abs{\w_0}W_\infty^-\,W_\infty)\le \varliminf_{n\to\infty}
 \E( \abs{\w_0}W_n)
\end{align*} 
while the last sequence is bounded, as can be seen by utilizing the 1-step decomposition \eqref{weak:W} and by taking $\beta$ in the interior of the region $\lambda(\beta)<\infty$.   
 Consequently 
\be\label{weak:H9}  
E^{\mu^h}[V]  - \beta^{-1} H(\mu^h\times q^h\vert\mu^h\times p)
=  \beta^{-1}(\lambda(\beta)+\kappa(\beta h) )  = \gpl^\beta(h).  
\ee
The pair $(\mu^h, q^h)$ is the unique one that satisfies \eqref{weak:H9}, by virtue
of the strict convexity of entropy.  

The maximizer for the point-to-point formula  \eqref{eq:var10.1} can also be found. 
Let $\gpp^\beta(\xi)$ be as in \eqref{p2pV_0} with $V_0(\w)=\w_0$.  
 Given $\xi\in\ri\Uset$,
$h\in\R^d$ can be chosen so that $\nabla\kappa(\beta h)=\xi$.  If $\beta$ is small
enough,  uniform integrability of the martingales 
$W_n$ can be ensured, and thereby $\mu^h$ and $q^h$ are again well-defined. 
The choice of $h$ implies that $E^{\mu^h}[Z_1]=\xi$, and we can turn \eqref{weak:H9}
into 
\begin{align*}
E^{\mu^h}[V_0]  - \beta^{-1} H(\mu^h\times q^h\vert\mu^h\times p)
&=  -  h\cdot E^{\mu^h}[Z_1] + \beta^{-1}(\lambda(\beta)+\kappa(\beta h) )  \\
& = \beta^{-1}\lambda(\beta)  - \beta^{-1}\kappa^*(\xi)= \gpp^\beta(\xi). 
\end{align*}
The last equality can be seen for example from duality \eqref{eq:tilt-velocity}.  

Markov chain    \eqref{qtheta}  appeared in  \cite{come-yosh-aop-06}. 
Under some restrictions on the environment and with $h=0$, 
  \cite{more-10} showed that    $\mu^0_0$  is the  limit of the environment seen   by the  particle.   \hfill $\triangle$ 

\end{example} 

\smallskip

We prove the theorems of this section, beginning with the positive temperature statements.

 \begin{proofof}{of Theorems \ref{th:var9} and \ref{th:var10}}
  Let   $V: \Omega_\ell\to\R$ be a member of $\cL$ 
 (Definition \ref{def:cL}), $\P$  ergodic and $0<\beta<\infty$.  
 Theorem 2.3 of \cite{Ras-Sep-Yil-13}   
gives the  variational formula  
	\begin{align}
	\Lapl^\beta&= \sup\Big\{ E^\mu[\min(V,c)]-\beta^{-1}\Hbar(\mu):\mu\in\cM_1(\Omega_\ell),\ c>0\Big\} .\label{eq:Lambda:H-var}
	\end{align}
 Note that \cite{Ras-Sep-Yil-13} used the uniform kernel $p(z)=\abs{\range}^{-1}$ but this makes no difference to the proofs, and in any case the kernel can be included in the potential to extend the result to an arbitrary kernel supported on $\range$.     We convert \eqref{eq:Lambda:H-var} into \eqref{eq:var9.1}  in a few steps.  

The measure  $\mu=\P\otimes\alpha$ with  $\alpha(z_{1,\ell})=p(z_{1,\ell})$
satisfies  $\mu\in\M_s(\Omega_\ell)$,  $\mu p=\mu$,  and $\Hbar(\mu)=0$.
Since $V(\cdot\,,z_{1,\ell})\in L^1(\P)$,
 this gives the finite lower bound  $\gpl^\beta\ge E^{\P\otimes\alpha}[V]$ for 
\eqref{eq:Lambda:H-var}.  (If $\ell=0$ the $\alpha$-factor is not there.)  Hence   we can  restrict
the supremum  in \eqref{eq:Lambda:H-var} to   $\mu$ such that $E^\mu[ V^- ] + \Hbar(\mu) < \infty$.  Since  $E^\mu[V]$ is   well-defined in $(-\infty,\infty]$  for all such  $\mu$, we can drop the 
truncation at  $c$.

 Entropy   has the following 
  representation:    for $\mu\in\M_1(\Omega_\ell)$, 
\be\label{entr-5}   
\inf_{q: \mu q=\mu}H(\mu\times q\,|\,\mu\times p)
=-\inf_{f\in b\Omega_\ell} E^\mu\big[\log\sum_zp(z)e^{f\circ S_z-f}\big].
\ee
The infimum on the left  is over Markov kernels $q$  on   $\Omega_\ell$ that fix $\mu$.  
$S_z$ is the shift mapping  defined in \eqref{eq:Sz}.
For a proof of \eqref{entr-5}  
see Theorem 2.1 of \cite{Don-Var-76}, Lemma 2.19 of \cite{Sep-93-ptrf-1},  or Theorem 14.2 of \cite{Ras-Sep-15-ldp}.

Recall the definition of $\Hbar$ in \eqref{eq:def:H}.  From the inequality 
\[\log\sum_zp(z)e^{f\circ S_z-f}\le \max_z \{f\circ S_z-f\}\le\log\sum_zp(z)e^{f\circ S_z-f}+\log\ellc^{-1}\]
follows, for $\mu_0\ll\P$,  
\begin{align}\label{eq:interesting}
\Hbar(\mu)-\log\ellc^{-1}\le -\inf_{f\in b\Omega_\ell} E^\mu\big[\max_z \{f\circ S_z-f\}\big]\le \Hbar(\mu).
\end{align}
If there exists   $f\in b\Omega_\ell$ such that 
$E^\mu[\max_z \{f\circ S_z-f\}]<0$
then replacing $f$ by $cf$ and taking $c\to\infty$ shows that  the infimum over $f$ is actually $-\infty$. This makes $\Hbar(\mu)=\infty$. 
Thus, relevant measures $\mu$ in \eqref{eq:Lambda:H-var} are ones that satisfy \eqref{eq:max>} and 
so  we   can insert the restriction $\mu\in\cM_s(\Omega_\ell)$ into \eqref{eq:Lambda:H-var}.   \eqref{eq:Lambda:H-var} has been converted into \eqref{eq:var9.1}.


 Assuming the directed i.i.d.~$L^{d+\e}$ setting described in Theorem \ref{th:ent2:lpp}, 
  Theorem 5.3 of \cite{Ras-Sep-14} gives
the point-to-point version:   for $\xi\in\Uset$,  
 \be \label{eq:rhs-var-p2p}
\Lapp^\beta(\xi)=\sup\bigl\{ E^\mu[\min(V,c)]-\beta^{-1}\Hbar(\mu):\mu\in\cM_1(\Omega_\ell),\ E^\mu[Z_1]=\xi,\ c>0\bigr\}.
 \ee
 This is converted  into \eqref{eq:var10.1} by the same reasoning used above. \hfill\hfill\qed
\end{proofof}


\begin{remark}\label{rmk:tight}  We can state 
 \eqref{eq:rhs-var-p2p} only for  the directed i.i.d.~$L^{d+\e}$ setting   for the following reason.  The point-to-level
formula \eqref{eq:Lambda:H-var} is proved directly in \cite{Ras-Sep-Yil-13}.
By contrast,   the point-to-point formula  \eqref{eq:rhs-var-p2p} is derived in \cite{Ras-Sep-14}
  via a contraction applied to a 
quenched large deviation principle (LDP) for polymer measures.   This LDP is proved in  \cite{Ras-Sep-Yil-13}.  
In the general setting  the upper bound of this LDP  has been proved only for compact 
sets  (weak LDP).  
However,  in the 
directed i.i.d.\ case the LDP 
is a full LDP, 
and the contraction works without additional assumptions.    Consequently in   the directed i.i.d.~$L^{d+\e}$ setting   \eqref{eq:rhs-var-p2p} is valid for Polish spaces $\Omega$,  but in the general setting $\Omega$ would need to be compact.  
   \hfill $\triangle$
 \end{remark}

 \begin{proofof}{of Theorems \ref{th:g=Hstar} and \ref{th:ent2:lpp}}
 Take $\beta\to\infty$ in \eqref{eq:var9.1} and \eqref{eq:var10.1}, utilizing bounds 
\eqref{H-9} and \eqref{eq:g-Lambda}. \hfill\hfill\qed
\end{proofof}

\begin{proofof}{of Proposition \ref{lm:S}}    Each  $f$ below  is a $b\Omega_\ell$   test function on  the appropriate space $\Omega_\ell$.   
  First we work with the case $\ell\ge1$.  We argue the implications (a.i)$\Rightarrow$(a.ii)$\Rightarrow$(a.iii)$\Rightarrow$(a.iv)$\Rightarrow$(a.i).

\medskip

(a.i)$\Rightarrow$(a.ii):  An $S$-invariant probability measure $\nu$ 
on 
  $\Omega_\N=\Omega\times\range^{\bN}$   that extends $\mu$ can be defined   by writing, for any $m\ge \ell$, 
 \be 
\int f(\w, z_{1,m})\,d\nu = \sum_{z_{1,m}} \int_\Omega  f(\w, z_{1,m})\,
\prod_{i=\ell+1}^m  q_{z_{i}}(T_{x_{i-\ell-1}}\w, z_{i-\ell, i-1}) \, \mu(d\w, z_{1,\ell}).  
\label{nu1}\ee

\medskip

(a.ii)$\Rightarrow$(a.iii):   From the $S$-invariance of $\nu$, 
\begin{align*}  &E^\mu\big[\max_z f(T_{Z_1}\w,(Z_{2,\ell},z))\big]=E^\nu\big[\max_z f(T_{Z_1}\w,(Z_{2,\ell},z))\big]\\
&\qquad 
=E^\nu\big[\max_z f(\w,(Z_{1,\ell-1},z))\big]  
\ge  E^\nu\big[  f(\w,Z_{1,\ell})\big] =E^\mu[f].  
\end{align*}


(a.iii)$\Rightarrow$(a.iv):  If $f$ is only a function of $(\w,z_{1,\ell-1})$, then $f(S_z(\w,z_{1,\ell}))=f(T_{z_1}\w,z_{2,\ell})$ does not  depend on $z$. 
\eqref{eq:max>} then implies  
$E^\mu\big[f(T_{Z_1}\w,Z_{2,\ell})]\ge E^\mu[ f ]$.  
Replacing $f$ by $-f$ makes this    an equality and \eqref{Sell-1}  follows.

\medskip

(a.iv)$\Rightarrow$(a.i):    Use property (a.iv)   in the second equality below to show that $\mu q=\mu$.  
\begin{align*}
&\int_{\Omega\times\range^{\ell}}  \sum_z  q_z(\w, z_{1,\ell}) f(T_{z_1}\w, (z_{2,\ell},z)) 
\, \mu(d\w, dz_{1,\ell})\\
&= \sum_z \int_{\Omega\times\range^{\ell}}    f(T_{z_1}\w, (z_{2,\ell},z)) 
\, \mu_\ell(z\,\vert\, T_{z_1}\w, z_{2,\ell})\, \mu(d\w, dz_{1,\ell})\\
&= \sum_z \int_{\Omega\times\range^{\ell}}    f(\w, (z_{1,\ell-1},z)) 
\, \mu_\ell(z\,\vert\, \w, z_{1,\ell-1})\, \mu(d\w, dz_{1,\ell})\\
  &=  \int_{\Omega\times\range^{\ell}}    f(\w, z_{1,\ell}) 
 \, \mu(d\w, dz_{1,\ell}) . 
\end{align*}

\medskip

We turn to the case $\ell=0$ and show    (b.i)$\Rightarrow$(b.ii)$\Rightarrow$(b.iii)$\Rightarrow$(b.i).


(b.i)$\Rightarrow$(b.ii):  Now define $\nu$ on 
  $\Omega\times\range^{\bN}$  by
\[E^\nu[f(\w,Z_{1,m})]=\sum_{z_{1,m}}\int f(\w,z_{1,m}) \prod_{i=1}^{m} q_{z_i}(T_{x_{i-1}}\w)\,\mu(d\w).\]


(b.ii)$\Rightarrow$(b.iii):   Analogously to (a.ii)$\Rightarrow$(a.iii) above, 
\begin{align*}  E^\mu\big[\max_z f(T_{z}\w)\big]&=E^\nu\big[\max_z f(T_{z}\w)\big] \ge  E^\nu\big[  f(T_{Z_1}\w)\big]=E^\nu[f(\w)] =E^\mu[f].  
\end{align*}

\medskip

(b.iii)$\Rightarrow$(b.i):   Observe that for   $f\in b\Omega$ we have  
\[E^\mu\big[\max_z \{f\circ T_z-f\}\big]\le E^\mu\big[\log\sum_zp(z)e^{f\circ T_z-f}\big]+\log\ellc^{-1}.\]
By  assumption  \eqref{eq:max>}   the left-hand side is nonnegative.
Then by \eqref{entr-5}
\begin{align*}
\inf\{H(\mu\times q\,|\,\mu\times p):\mu q=\mu\}
=-\inf_{f\in b\Omega} E^\mu\big[\log\sum_zp(z)e^{f\circ T_z-f}\big]\le\log\ellc^{-1}. 
\end{align*}
Since the infimum is not $+\infty$  there must exist a Markov kernel $q$ that fixes $\mu$ and for which $H(\mu\times q\,|\,\mu\times p)<\infty$.
This implies that for $\mu$-a.e.\ $\w$ the kernel is supported on $\{T_z\w:z\in\range\}$.  \hfill\hfill\qed
\end{proofof} 

\section{Periodic environments} \label{sec:finite}

The case of finite $\Omega$ provides explicit  illustration of the theory
developed in the paper.   The point-to-level limits and  solutions to the  variational formulas
come from  Perron-Frobenius theory, the classical  theory for $0<\beta<\infty$
and the max-plus theory for $\beta=\infty$. 
(See \cite{bacc-cohe-etal-book, berm-plem-book, heid-olds-woud, sene-book} for expositions.)
    We consider a   potential 
$V(\w,z)=V_0(\w)+h\cdot z$ for $(\w,z)\in\Omega\times\range$, $h\in\R^d$.

Let  $\Omega$ be a finite set of  $m$ elements.  As all along, 
  $\{T_x\}_{x\in\gr}$ is a group of commuting bijections  on $\Omega$ that act irreducibly. 
 That is, for each pair $(\w,\w')\in\Omega\times\Omega$ there exist $z_1,\dotsc, z_k\in\range$ such that  $T_{z_1+\dotsm+ z_k}\w=\w'$.  
The ergodic probability measure
is $\P(\w)=m^{-1}$. 

A basic example is a periodic environment indexed by $\Z^d$.  
Take a vector $a>0$ in $\Z^d$ (coordinatewise inequalities),
define the rectangle  $\rectan=\{ x\in\Z^d:  0\le x<a \}$,  
  fix a finite configuration $(\bar\w_x)_{x\in\rectan}$, and 
then extend $\bar\w$ to all of $\Z^d$  periodically:  $\bar\w_{x+k\circ a}=\bar\w_x$ for 
$k\in\Z^d$, where $k\circ a=(k_ia_i)_{1\le i\le d}$ is the coordinatewise product of 
two vectors.   
Irreducibility holds for example if $\range$ contains $\{e_1,\dotsc, e_d\}$.  

\subsection{Case $0<\beta<\infty$}  
We take $\beta=1$ and drop it from the notation.  
 Define a nonnegative irreducible  matrix indexed by $\Omega$ by  
\be\label{pf-R}   \amatri_{\w, \w'}=   \sum_{z\in\range}p(z)\,  \one\{ T_z\w=\w'\,\}e^{V_0(\w)+h\cdot z}  \quad \text{for $\w,\w'\in\Omega$.}
 \ee
Let $\rho$ be the Perron-Frobenius eigenvalue (spectral radius) of $\amatri$.   Then by standard asymptotics   the  
 limiting point-to-level free energy is
\be\begin{aligned}
\gpl(h) &= \lim_{n\to\infty} n^{-1}\log \sum_{x_{0,n}: \,x_0=0} p(x_{0,n}) 
e^{\sum_{k=0}^{n-1} V_0(T_{x_k}\w) + h\cdot x_n}  \\
&=\lim_{n\to\infty} n^{-1}\log \sum_{\w'\in\Omega} \amatri^n_{\w, \w'}
=\log\rho. 
\end{aligned}\label{pf:gpl}\ee

On a finite $\Omega$ every cocycle is a gradient (proof left to the reader).   Hence we can replace the general cocycle $F$ with a gradient $F(\w, 0, z)=f(T_z\w)-f(\w)$ and write the cocycle variational formula   \eqref{eq:Lambda:K-var} as  
\be\label{pf:var-h}  
\gpl(h)=\inf_{f\in\R^\Omega} \,\max_\w\;  \log \sum_{z\in\range} p(z)e^{\Vw(\w)+  h\cdot z+f(T_z\w)-f(\w)}.  
\ee
This is now exactly the same as the following   textbook characterization of the Perron-Frobenius eigenvalue:  
\be\label{pf:char1}   \rho = \inf_{\varphi\in\R^\Omega:\,\varphi>0} \;   \max_\w  \; \frac1{\varphi(\w)} \sum_{\w'}  A_{\w, \w'} \varphi(\w').  \ee

 Let  $\lev$ and $\rev$ be the  left and 
right (strictly positive)  Perron-Frobenius eigenvectors  of $A$ normalized so that 
$\sum_{\w\in\Omega} \lev(\w)\rev(\w)=1$.  For each  $\w\in\Omega$ the left eigenvector equation is 
\be\label{lev}  \sum_{z\in\range}p(z)\,  e^{V_0(T_{-z}\w)+h\cdot z} \lev(T_{-z}\w)=\rho\lev(\w)
\ee
and the right eigenvector equation  is 
\be\label{rev}  e^{V_0(\w)}\sum_{z\in\range} p(z) e^{h\cdot z} \rev(T_z\w)=\rho\rev(\w). 
\ee
The right eigenvector equation \eqref{rev}  says   that the gradient  
\be\label{pf:F} F(\w,x,y)=\log\rev(T_y\w)-\log\rev(T_x\w)\ee
minimizes in \eqref{pf:var-h} without
the maximum over $\w$  (the right-hand side of  \eqref{pf:var-h} is  constant in $\w$).   In other words, $F$ is a corrector for $\gpl(h)$.    Compare this to \eqref{eq:Kvar:minbeta}. 

Define a probability measure on $\Omega$ by $\mu_0(\w)=\lev(\w)\rev(\w)$.  
The  left eigenvector equation \eqref{lev}  says that $\mu_0$ is invariant
under the stochastic kernel 
\be\label{pf-q}  
q_0(\w,\w')=  \sum_{z\in\range}p(z)\,  \one\{ T_z\w=\w'\}e^{V_0(\w)+h\cdot z+ F(\w,0,z)-\gpl(h)}, \quad \w, \w'\in\Omega. 
\ee
Using this one can check that the measure 
\[  \mu(\w, z_1) = p(z) {\mu_0(\w)}    
e^{V_0(\w)+h\cdot z_1+ F(\w,0,z_1)-\gpl(h)}
\]
is a member of  $\M_s(\Omega\times\range)$ 
   and invariant under the kernel  
\[q( (\w, z_1),  (T_{z_1}\w, z)) = 
 p(z)   e^{V_0(T_{z_1}\w)+h\cdot z+ F(T_{z_1}\w,0,z)-\gpl(h)}. \]
Another  computation checks that 
\[  E^\mu[V_0(\w)+h\cdot Z_1]- H(\mu\times q\,\vert \,\mu\times p) 
= \gpl(h). \]
Hence $\mu$ is a maximizer in the entropy variational formula \eqref{eq:Lambda:H-var}. 



Assume additionally  that matrix $\amatri$ is aperiodic on $\Omega$.  Then $A$ is primitive, that is,  $\amatri^n$ is strictly positive  for large enough $n$.    Perron-Frobenius asymptotics (for example, Theorem 1.2  in \cite{sene-book}) give the   Busemann function 
$\Bpl^h$ of \eqref{buse4}.  
\begin{align*}
& \Bpl^h(\w,0,z)
=\lim_{n\to\infty} \biggl\{  \log \sum_{x_{0,n}: \, x_0=0} p(x_{0,n})   
e^{\sum_{k=0}^{n-1}V_0(T_{x_k}\w)+h\cdot x_n}  \\
&\qquad\qquad \qquad\qquad 
\;-\;  
\log \sum_{x_{0,n-1}: \,x_0=z} p(x_{0,n-1})   
e^{\sum_{k=0}^{n-2}V_0(T_{x_k}\w)+h\cdot (x_{n-1}-z)} \biggr\} \\
& =\lim_{n\to\infty} \biggl\{ \log \sum_{\w'\in\Omega} \amatri^n_{\w, \w'}  - \log \sum_{\w'\in\Omega} \amatri^{n-1}_{T_z\w, \w'}\biggr\}\\
&=\lim_{n\to\infty} \biggl\{   \log\rho + \log \Bigl( \sum_{\w'\in\Omega} \rev(\w)\lev(\w') + o(1)\Bigr) 
-  \log \Bigl( \sum_{\w'\in\Omega} \rev(T_z\w)\lev(\w') + o(1)\Bigr)\biggr\} \\
&=
\log\rho + \log   \rev(\w) -  \log   \rev(T_z\w) .  
\end{align*}

  If we assume that all admissible paths between two given points   have the same number of 
steps, then  $ \Bpl^h(\w,0,z)$ extends to a stationary $L^1$ cocycle,  as showed  in
Theorem \ref{th:Bus=grad(b)}.   Then this  situation fits the   development  of Sections \ref{sec:corr}--\ref{sec:bus}.   Equation \eqref{rev} shows that cocycle 
\be\label{pf:B5} \wt B(\w,0,z)=\Bpl^h(\w,0,z)-h\cdot z\ee
 is adapted to 
 $\Vw$, illustrating  Theorem \ref{th:Bus=grad(b)}.   Definition \eqref{EB} applied to   the explicit formulas above gives  
\[   h(\wt B)\cdot z=-\E[\wt B(\w,0,z)]=   -\log \rho+ h\cdot z
\qquad \text{ for each $z\in\range$.  }  \]  
Consequently 
$h(\Bpl^h)\perp\aff\range$.    By Theorem \ref{thm:minimizer}  the cocycle 
\begin{align}\label{pf:F5} 
\begin{split}
\wt F(\w, 0,z)&=-h(\wt B)\cdot z - \wt B(\w,0,z)=\log \rho-\Bpl^h(\w, 0,z)\\
&=\log\rev(T_z\w)-\log\rev(\w), 
\end{split}
\end{align} 
that appeared in \eqref{pf:F},   is the minimizer in \eqref{pf:var-h} for any tilt $ h'$ such that 
$(h'-h(\wt B))\cdot z=   (h'-h)\cdot z +\log \rho $ is constant over $z\in\range$.  

Connection \eqref{pf:gpl} between the  limiting free energy and the Perron-Frobenius eigenvalue is standard fare in textbook treatments of the large deviation theory of finite Markov chains
\cite{demb-zeit, stro-ldp-84, Ras-Sep-15-ldp}.  
 
\subsection{Point-to-level last-passage case}

The {\sl max-plus algebra}  is the semiring $\R_{\text{max}}$ $=$ $\R\cup\{-\infty\}$ under the 
operations 
$ x\oplus y=x\vee y $  and   $x\otimes y=x+y$. 
 Define an irreducible  $\R_{\text{max}}$-valued matrix   by  
\be\label{pf-R2}   \bmatri(\w, \w')= \begin{cases}  \ddd V_0(\w)+\max_{z:T_z\w=\w'} h\cdot z,   &\w'\in\{T_z\w: z\in\range\} \\
-\infty, &\w'\notin\{T_z\w: z\in\range\} .\end{cases}   \ee
 As an irreducible  matrix $A$ has a unique  finite max-plus  eigenvalue $\lambda$ together with a (not necessarily  unique even up to an additive constant) finite eigenvector $\sigma$ that satisfy  
\be\label{pf:max+4}   \max_{\w'\in\Omega}\, [ A(\w,\w')+\sigma(\w')] = \lambda + \sigma(\w), \qquad \w\in\Omega. \ee  
Inductively 
\be\label{pf:max+5}   \max_{\w=\w_0,\,\w_1,\dotsc,\,\w_n} \Bigl\{  \; \sum_{k=0}^{n-1}  A(\w_k,\w_{k+1})  +\sigma(\w_n) \Bigr\} =n \lambda + \sigma(\w), \qquad \w\in\Omega. \ee 

 The last-passage value from \eqref{p2lh} can be expressed as   
\begin{align} 
\label{pf:max+5.5}   \Gpl_{0,(n)}^\infty(h)&=  
 \max_{x_{0,n}} \sum_{k=0}^{n-1}\bigl( \Vw(T_{x_k}\w) + h\cdot (x_{k+1}-x_k) \bigr)    \\
\nn  &=\max_{\w=\w_0,\,\w_1,\dotsc,\,\w_n}   \sum_{k=0}^{n-1} A(\w_k,\w_{k+1}) . 
\end{align} 
Dividing through \eqref{pf:max+5} by $n$ gives the limit 
\begin{align*}
\gpl^\infty(h)=  \lim_{n\to\infty} n^{-1}\Gpl_{0,(n)}^\infty(h)   = \lambda.  
\end{align*}
The eigenvalue equation \eqref{pf:max+4} now  rewrites as 
\be\label{pf:max+6}  
\gpl^\infty(h)=  \max_{z\in\range} \{\Vw(\w)+ h\cdot z+ \sigma(T_z\w)-\sigma(\w) \}. 
\ee
This  is the cocycle variational formula \eqref{eq:g:K-var} (without the supremum over $\w$) and shows  that a corrector  is given by the gradient 
\be\label{pf:F13} F(\w,0,z)= \sigma(T_z\w)-\sigma(\w). \ee 
Compare \eqref{pf:max+6} to \eqref{eq:Kvar:min}. 

The measure variational formula \eqref{eq:var9.1} links with an alternative characterization of 
the max-plus eigenvalue   as the maximal average weight of an elementary circuit. To describe this,  consider the  directed graph $(\Omega, \cE)$ 
  with vertex set $\Omega$ 
and  edges $\cE=\{(\w,T_z\w):\w\in\Omega, z\in\range\}$.    
This allows multiple edges from $\w$ to $\w'$  and   loops from $\w$ to itself.
Loops happen in particular if  $0\in\range$.  Identify edge $(\w,T_z\w)$  with the pair $(\w, z)$.  
An {\sl elementary circuit}  of length $N$   is a sequence of edges $(\w_0, z_1), (\w_1, z_2), \dotsc, (\w_{N-1}, z_N)$ such that $\w_i=T_{z_i}\w_{i-1}$ with  $\w_N=\w_0$, but  $\w_i\ne \w_j$ for $0\le i<j<N$.  

Given any fixed $\w$,   all elementary circuits can be represented as  admissible paths   $x_0, x_1$, $\dotsc$, $x_N$ in $\gr$ by choosing $x_0$ so that  $\w_0=T_{x_0}\w$ and $x_i=x_{i-1}+z_i$ for $1\le i\le N$.    Conversely, an admissible path   $x_0, x_1, \dotsc, x_N$ in $\gr$ represents an  elementary circuit if $T_{x_0}\w, T_{x_1}\w, \dotsc, T_{x_{N-1}}\w$ are distinct,   but $T_{x_0}\w=T_{x_N}\w$.   Let $\cC$ denote  the set of elementary circuits. 
 The   average weight  formula for the eigenvalue  is (Thm.~2.9 in \cite{heid-olds-woud})
\be\label{pf:aa} 
\lambda= \max_{N\in\N, \, x_{0,N}\in\cC}   N^{-1}\sum_{k=0}^{N-1}
\bigl(  V_0(T_{x_k}\w)+h\cdot z_{k+1}\bigr). 
\ee
 The right-hand side   is independent of $\w$ because switching $\w$ amounts to  
translating the circuit, by the assumption of irreducible action by $\{T_z\}_{z\in\range}$.  

 It is elementary to verify from definitions that  $\gpl^\infty(h)$ equals  the right-hand side of \eqref{pf:aa}.  (The sum on the right-hand side of \eqref{pf:max+5.5}  decomposes into circuits and a bounded part,  while an asymptotically optimal path finds a maximizing circuit and repeats it forever.)    
  If we take    \eqref{pf:aa} as the definition of $\lambda$,  then the identity 
 \be\label{pf:max+9} \lambda = \max\Bigl\{\, \sum_{(\w,z) \in\Omega\times\range}\mu(\w,z)(V_0(\w) +h\cdot z) :  \mu\in\cM_s(\Omega\times\range)\Bigr\}   \ee
follows from the fact that the extreme points of  the convex set  $\cM_s(\Omega\times\range)$ are exactly those  uniform probability measures whose support is a single  elementary circuit. We omit the proof.   Equation \eqref{pf:max+9} is   the measure variational formula \eqref{eq:var9.1} which has now been (re)derived in the finite setting from max-plus theory.  

 As in the finite temperature case, existence of point-to-level Busemann functions follows from asymptotics of matrices.    The {\sl critical graph} of the  max-plus matrix  $A$ is the subgraph of $(\Omega, \cE)$ consisting of those nodes and edges that belong to elementary circuits that maximize in \eqref{pf:aa}.   
 Matrix  $A$  is {\sl primitive} if it is irreducible and if its critical graph has a unique strongly connected component with cyclicity $1$ (that is, a unique  irreducible and aperiodic component  in Markov chain terminology).   
 This implies that the eigenvector is unique up to an additive constant  and these asymptotics hold as $n\to\infty$: 
  \be
  \begin{split}
     \Gpl_{0,(n)}^\infty(h)-  \Gpl_{z,(n-1)}^\infty(h) &= (A^{\otimes n}\otimes\mathbf0)(\om) - (A^{\otimes (n-1)}\otimes\mathbf0)(T_z\om) \\
   	&\phantom{xxxxx}\longrightarrow  \lambda +  \sigma(\w) - \sigma(T_z\w)\equiv \Bpl^h(\w,0,z). \label{eq:per:p2l:bus}
    \end{split}
    \ee
(From \cite{heid-olds-woud} apply  Thm.~3.9 with cyclicity 1 and section 4.3.) Above $\mathbf0 = (0,\ldots,0)^T$ and operations $\otimes$ are in the max-plus sense.  
 Equation \eqref{pf:max+6} shows that cocycle 
$\wt B(\w,0,z)=\Bpl^h(\w,0,z)-h\cdot z$ is adapted to 
$\Vw$,   as an example of Theorem \ref{th:Bus=grad(b)} for $\beta=\infty$.     

\medskip

The next simple example illustrates the max-plus case.  
All the previous results of this paper  identify correctors that solve the variational formulas of Theorem \ref{th:K-var} so that   the essential supremum over $\w$ can be dropped.   This  example shows that there can be additional minimizing cocycles $F$ for which the function of $\w$  on the right in \eqref{eq:g:K-var} is not constant in $\w$. 

\medskip

\begin{figure}[h]  
\begin{center}
\begin{tikzpicture}[>=latex,yscale=0.6,xscale=0.8]
	\foreach \x in {0,2,4,6}
			{
				\foreach \y in {0,...,3}	
					{
						\draw(\x,0)--(\x,4);
						\draw(\x+1,0)--(\x+1,4);
						\draw(0,\y)--(8,\y);
						\draw (\x +.5,\y+.5) node{\Large{$1$}};
						\draw(\x+1.5, \y+.5) node{\Large{$0$}};
						\draw(3.5, \y+.5) node{\Large{$0$}};
					}
				 
			}
	\draw[color=nicosred, line width=4pt](4,1)rectangle(5,2);
	\fill[color=orange, nearly transparent](4,1)rectangle(5,2);
	\draw[color=white, line width=2pt](0,0)--(0,4);
	\draw[color=white, line width=2pt](0,0)--(8,0);
	
\end{tikzpicture}
\end{center}
\caption{Environment configuration $\w^{(1)}$ indexed by $\Z^2$ in Example \ref{ex:stripes1}. The origin is shaded in a thick frame. }
 \label{fig:tile}
\end{figure}

\begin{example} \label{ex:stripes1}
Take $d=2$ and a two-point environment space $\Omega=\{\w^{(1)}, \w^{(2)}=T_{e_1}\w^{(1)}\}$  where $\w^{(1)}_{i,j}=\tfrac12(1+(-1)^i)$ for  $(i,j)\in\Z^2$ is a vertically striped configuration of 
  zeroes and ones, with a one at the origin (Figure \ref{fig:tile}). Admissible steps are $\range=\{e_1, e_2\}$ and $T_{e_2}$ acts as an identity.     The ergodic measure is 
$\P=\tfrac12(\delta_{\w^{(1)}}+ \delta_{\w^{(2)}})$ and 
the potential    $V_0(\w)=\w_0$ with tilts $h=(h_1, h_2)\in\R^2$.    

 Matrix $A(\w^{(i)}, \w^{(j)})$ of \eqref{pf-R2} is 
		\[ A= 
		\begin{bmatrix}
			1+h_2		&1+ h_1\\
			h_1			& h_2\\
		\end{bmatrix}
		\]
and  the  directed graph $(\Omega, \cE)$ 	is in Figure \ref{pf:gr-fig}.

	\begin{figure}[h]  
	\begin{center}
	\begin{tikzpicture}[>=latex, yscale = 0.7]
	\draw (0,0)node{$\om^{(2)}$};
	\draw[color=nicosred, line width= 2pt] (0,0)circle(.5cm);
	\draw (5,0)node{$\om^{(1)}$};
	\draw [color=nicosred, line width= 2pt](5,0)circle(.5cm);
	
	\draw[<-] (0.4,.5)..controls(1.5,1.5) and (3.5,1.5)..(4.6,.5);
	\draw[->] (0.4,-.5)..controls(1.5,-1.5) and (3.5,-1.5)..(4.6,-.5);
	
	\draw[->](-0.6,0.2)..controls (-1.7,1.3) and (-1.7,-1.3)..(-0.6,-0.2); 
	\draw[->](5.6,0.2)..controls (6.7,1.3) and (6.7,-1.3)..(5.6,-0.2); 
	
	\draw(2.5, 1.2)node[above]{$1+h_1$};
	\draw(2.5, -1.2)node[below]{$h_1$};
	\draw(-1.5, 0)node[left]{$h_2$};
	\draw(6.5, 0)node[right]{$1+h_2$};
	
	\end{tikzpicture}	 \end{center}  
	\caption{Graph $(\Omega, \cE)$  for Example \ref{ex:stripes1}.} 	
	 \label{pf:gr-fig}
	\end{figure}

Since $A$  is irreducible   its unique max-plus 
	eigenvalue is  the maximum average value of   elementary circuits and this gives the point-to-line last-passage limit: 
		\be
		\label{eq:lambda:R}
			\gpl^{\infty}(h)= \lambda = \max\{  \tfrac12 +h_1,  1+h_2 \}.
		\ee
	There are two cases to consider, and in both cases there is a unique  eigenvector (up to an additive constant)  $\sigma=(\sigma(\w^{(1)}), \sigma(\w^{(2)}))$:
		\begin{enumerate}
		\item[(i)] $ \tfrac12+ h_1  \le 1+ h_2= \lambda$, $\sigma=(1, h_1-h_2)$,  the critical graph has cyclicity 1. 
			\item[(ii)] $1+h_2 < \tfrac12+ h_1  = \lambda$, $\sigma=(1, \tfrac12)$,  the critical graph has cyclicity 2.
		\end{enumerate}

\smallskip

{\sl Case {\rm(}i{\rm)}.}  One can verify by hand that  variational formula \eqref{eq:g:K-var} is minimized by the   cocycles 
\be\label{se:F9}   F(\w^{(1)}, 0, e_1)=a=-F(\w^{(2)}, 0, e_1), \quad
F(\w^{(1)}, 0, e_2)=F(\w^{(2)}, 0, e_2)=0
\ee
for $a\in[h_1-h_2-1, h_2-h_1]$.  Let $\wt F$ denote  the cocycle for $a=h_1-h_2-1$ which is the one consistent with \eqref{pf:F13} for the eigenvector $\sigma$.     Among the minimizing  cocycles  only $\wt F$  satisfies \eqref{eq:g:K-var} without $\max_\w$,  that is, in the form \eqref{eq:Kvar:min}.  And indeed this corrector comes from Theorem \ref{thm:minimizer}(ii-b).  $\wt F$ is given by  equation \eqref{FF} with  a cocycle $\wt B$ that is adapted to 
$\Vw$ (as defined in \eqref{VB}) if and only if  $1+h_2\ge \tfrac12+h_1$.   In case (i)  matrix  $A$ is primitive  and  limit \eqref{eq:per:p2l:bus}  gives an explicit  Busemann function $ \Bpl^h(\w,0,z)$.  From this Busemann function   \eqref{pf:B5}  gives cocycle $\wt B$.    

\smallskip 

{\sl Case {\rm(}ii{\rm)}.}  In this case there is a unique minimizing corrector $\check F$ which is \eqref{se:F9} with $a=-1/2$, the one that satisfies \eqref{pf:F13} for the eigenvector $\sigma$.  $\check F$ comes   via  equation \eqref{FF} from a cocycle  that is adapted to 
$\Vw$  if and only if  $ \tfrac12+h_1\ge 1+h_2$.  So the variational formula  \eqref{eq:g:K-var} is again satisfied without $\max_\w$.  
However, this time $\check F$ cannot come from  Busemann functions because 
 some Busemann functions do not exist.  Maximizing $n$-step paths use only $e_1$-steps and consequently 
\[  
 \Gpl_{0,(n)}^\infty(h)-  \Gpl_{e_2,(n-1)}^\infty(h) = h_1 + \one\{\text{$n$ is odd}\}  \]
does not converge as $n\to\infty$.  

Note that $\check F$ is a minimizing cocycle in both cases (i) and (ii), but only in case (ii)   it  
 satisfies \eqref{eq:g:K-var} without $\max_\w$.   
   \hfill$\triangle$
  \end{example}

\appendix

 \section{Auxiliary lemmas}\label{sec:cK}

Centered cocycles satisfy a uniform ergodic theorem. The following is a special case of Theorem 9.3 of  \cite{geor-rass-sepp-yilm-15}. Note that a one-sided bound suffices for a hypothesis.   Recall Definition \ref{def:cL} for class $\cL$ and Definition \ref{def:cK} for the space $\cK_0$ of centered cocycles.   

\begin{theorem}\label{th:Atilla}
Assume $\P$ is ergodic under the transformations  $\{T_z:z\in\range\}$. 
Let $F\in\cK_0$.   Assume there exists  $V\in\cL$ such that $\max_{z\in\range} F(\w,0,z)\le V(\w)$ for $\P$-a.e.\ $\w$.  
Then for $\P$-a.e.\ $\w$
\[\lim_{n\to\infty}\;\max_{\substack{x=z_1+\dotsm+z_n\\z_{1,n}\in\range^n}} \;\frac{\abs{F(\w,0,x)}}n=0.\]
\end{theorem}

\begin{lemma}\label{app:lm1} 
Let  $X_n\in L^1$, $X_n\to X$ a.s.,     $\ddd\varliminf_{n\to\infty}   EX_n\le c<\infty$, and $X_n^-$ uniformly integrable. Then $ X\in L^1$ and $EX\le c$. 
\end{lemma}
\begin{proof}
   Since $X_n^-\to X^-$ a.s. and $X_n^-$  is uniformly integrable,   $X_n^-\to X^-$ in $L^1$ and in particular $X^-\in L^1$.   By Fatou's lemma and by the assumption,  
\begin{align*}
E(X^+)&=E(\lim_{n\to\infty} X_n^+) \le  \varliminf_{n\to\infty}  E(X_n^+) =   \varliminf_{n\to\infty}  E(X_n + X_n^-) 
\le c + E(X^-) < \infty  
 \end{align*}
 from which we conclude that $X\in L^1$ and then $EX\le c$. \hfill\hfill\qed
\end{proof}


%

\footnotesize

\bibliographystyle{plain}

\bibliography{firasbib2010,growthrefs}

\end{document}